\documentclass[12pt]{amsart}

\usepackage{amsmath,amssymb,fullpage}

\usepackage[bbgreekl]{mathbbol}

\usepackage{graphicx}
\usepackage{graphicx}
\usepackage{amssymb}
\usepackage{epstopdf}
\usepackage{tikz}
\usepackage[all,cmtip]{xy}
\usepackage{color}
\usepackage[parfill]{parskip}

\theoremstyle{plain} 
\newtheorem{thm}{Theorem} 
\newtheorem{cor}[thm]{Corollary} 
\newtheorem{lem}[thm]{Lemma} 
\newtheorem{conj}[thm]{Conjecture} 
\newtheorem{prop}[thm]{Proposition} 
\theoremstyle{definition} 
\newtheorem{defn}[thm]{Definition}

\renewcommand{\int}{\operatorname{int}}

\newcommand{\Ker}{\operatorname{Ker}}
\newcommand{\Cok}{\operatorname{Cok}}
\newcommand{\id}{\operatorname{id}}

\newcommand{\br}{\mathbb B\text{r}}

\newcommand{\im}{\operatorname{Im}}

\newcommand{\Z}{\mathbb{Z}}
\newcommand{\z}{ \mathbb Z_2}

\newcommand{\out}{\operatorname{Out}}

\newcommand\sL{\text{\sf L}}
\newcommand\sD{\text{\sf D}}

\newcommand{\cT}{\mathcal{T}}

\newcommand{\zh}{\mathbb Z[\frac{1}{2}]}

\newcommand{\tree}[3]{\text{\Large {$ {\text{\normalsize $ #1$}}-\!\!\!<^{#2}_{#3}$}}}

\begin{document}

\title{Tree Homology and a Conjecture of Levine}

\begin{abstract}
In his study of the group of homology cylinders, J. Levine \cite{L2} made the conjecture that a certain homomorphism $\eta'\colon\mathcal T\to{\sf D}'$ is an isomorphism. Here $\cT$ is an abelian group on labeled oriented trees, and $\sD'$ is the kernel of a bracketing map on a quasi-Lie algebra. Both $\cT$ and $\sD'$ have strong connections to a variety of topological settings, including the mapping class group, homology cylinders, finite type invariants, Whitney tower intersection theory, and the homology of $\out(F_n)$.
In this paper, we confirm Levine's conjecture. This is a central step in classifying the structure of links up to grope and Whitney tower concordance, as explained in other papers of this series \cite{CST0,CST1,CST2,CST4}. We also confirm and improve upon Levine's conjectured relation between two filtrations of the group of homology cylinders.
\end{abstract}

\author[J. Conant]{James Conant}
\email{jconant@math.utk.edu}
\address{Dept. of Mathematics, University of Tennessee, Knoxville, TN 37996}

\author[R. Schneiderman]{Rob Schneiderman}
\email{robert.schneiderman@lehman.cuny.edu}
\address{Dept. of Mathematics and Computer Science, Lehman College, City University of New York, Bronx, NY 10468 }

\author[P. Teichner]{Peter Teichner} 
\email{teichner@mac.com} 
\address{Dept. of Mathematics, University of California, Berkeley, CA and} 
\address{Max-Planck Institut f\"ur Mathematik, Bonn, Germany} 



\maketitle

\section{Introduction}
There is an interesting Lie algebra ${\sf D}(H)$ that lies at the heart of many areas of topology. It can be defined as the kernel of the bracketing map $H\otimes {\sf L}(H)\to {\sf L}(H)$ on the free Lie algebra ${\sf L}(H)$ over a symplectic vector space $H$. An equivalent definition is that ${\sf D}(H)={\sf Der}_\omega({\sf L}(H))$ is the set of derivations of the free Lie algebra that kill the symplectic element $\omega=\sum_i[p_i,q_i]$  where $\{p_i,q_i\}$ is a symplectic basis of $H$.

When $H=H_1(\Sigma_g;\Z)$ is the first homology of a closed oriented surface $\Sigma_g$ of genus $g$,
Dennis Johnson used ${\sf D}(H)$ to study the relative weight filtration of the mapping class group of the surface $\Sigma$ \cite{J}. Specifically, he showed that the associated graded group is a Lie algebra which embeds in ${\sf D}(H)$. In a different direction, letting $H$ be the direct limit of finite dimensional symplectic vector spaces, Kontsevich first noticed that the homology of the Lie algebra ${\sf D}(H)$ can be used to study the rational homology of outer automorphism groups of free groups, a beautiful connection that was exploited by Morita and then by Conant-Vogtmann \cite{CV,CV2,Kon,Mor2}.

When $H$ is an abelian group with no symplectic structure, one can still define an abelian group ${\sf D}(H)$ as the kernel of the bracketing map. Letting $H$ be the first homology of a link complement,
${\sf D}(H)$ becomes the natural home for Milnor invariants of the link \cite{CST2,HM,O}.

In \cite{L1}, Levine clarified Johnson's construction by enlarging the mapping class group to a group of homology cylinders, proving that the associated graded group becomes all of ${\sf D}(H)$.  In the context of homology cylinders there is another natural filtration, introduced by Habiro and Goussarov \cite{GGP,Gouss,Hab}, which is related to finite type invariants, called the $Y$-filtration. In order to relate the $Y$-filtration to the relative weight filtration,
Levine worked with a map $\eta_n$ from an abelian group of trivalent trees with $n$ trivalent vertices, 
$\mathcal T_n$ (defined below), to ${\sf D}_n$. (Close relations of the group  $\mathcal T_n$ have previously appeared in the the theory of Goussarov-Vassiliev invariants \cite{HM,Mo}. See also \cite{Hab1} for more connections to homology cylinders.)
Rationally, $\eta_n$ was known to be an isomorphism, and initially Levine thought that it was an isomorphism over the integers as well. However, in \cite{L2} he published a correction, noting that certain symmetric trees were in the kernel of $\eta_n$, but nevertheless establishing that $\eta_n$ is onto. In order to promote $\eta_n$ to an isomorphism, Levine realized that in the definition of ${\sf D}_n$ as the kernel of a bracketing map,  one needs to replace free Lie algebras by free \emph{quasi}-Lie algebras. These are defined like Lie algebras, but with the self-annihilation relation $[Z,Z]=0$ replaced by antisymmetry $[Y,Z]=-[Z,Y]$. This leads to a new map $\eta'_n\colon\mathcal T_n\to{\sf D}'_n$, and Levine made what we have been calling the \emph{Levine Conjecture}, that this is an isomorphism, proving it in many special cases.

This paper is third in a series of papers on Whitney tower intersection theory, in which the
 group $\mathcal T_n$ is of central importance.  The question of whether $\eta'$ is an isomorphism plays a crucial role in our arguments. For example, establishing that $\eta'$ is an isomorphism leads to the classification theorems for geometric filtrations of link concordance as described in \cite{CST0,CST1,CST2,CST4}. Thus we were motivated to prove Levine's conjecture:

\begin{thm}\label{mainthm}
 $\eta'_n\colon\mathcal T_n\to{\sf D}'_n$ is an isomorphism for all $n$.
\end{thm}

As discussed in Section~\ref{sec:filtrations}, this also allows us to prove (Theorem~\ref{thm:filtration}) and improve upon (Theorem \ref{thm:filt}, which we will prove in \cite{CST5}) Levine's conjectured relationship between the associated graded groups of homology cylinders with respect to the relative weight filtration and the Goussarov-Habiro $Y$-filtration \cite{L2}.

Fix an index set $\{1,\ldots,m\}$ once and for all. Let $\mathcal T_n=\mathcal T_n(m)$ be the abelian group of formal linear combinations of oriented unitrivalent trees with $n$ internal vertices (with no distinguished root), where the univalent vertices are labeled by elements of the index set, modulo IHX relations and antisymmetry relations. (In Levine's notation, $\mathcal T_n=\mathcal A^t_n(H)$, where $H\cong \mathbb Z^m$ is the free $\mathbb Z$-module spanned by the index set.) 

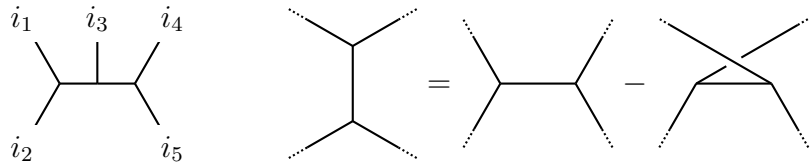
\begin{figure}[h]
\begin{center}
\begin{minipage}{1.5cm}
    \begin{tikzpicture} [thick]
\draw (0,0) -- (-.5,.866) node [fill=white] {$i_1$};
\draw (0,0) -- (-.5,-.866) node [fill=white] {$i_2$};
\draw (0,0) -- (.5,0);
\draw (.5,0)--(.5,.866) node [fill=white] {$i_3$};
\draw (.5,0)--(1.0,0);
\draw (1.0,0)--(1.5,.866) node [fill=white] {$i_4$};
\draw (1.0,0)--(1.5,-.866) node [fill=white] {$i_5$};
\end{tikzpicture}
\end{minipage}
\hspace{.8in}
\begin{minipage}{1.732cm}
\begin{tikzpicture} [thick,rotate=90]
\draw (0,0) -- (-.35,.6062);
\draw[densely dotted] (-.35,.6062)--(-.5,.866);
\draw (0,0) -- (-.35,-.6062);
\draw[densely dotted] (-.35,-.6062)-- (-.5,-.866);
\draw (0,0) -- (1,0);
\draw (1,0) -- (1.35,.6062);
\draw[densely dotted] (1.35,.6062)--(1.5,.866);
\draw (1,0) -- (1.35,-.6062);
\draw[densely dotted] (1.35,-.6062)--(1.5,-.866);
\end{tikzpicture}
\end{minipage}
$=$
\begin{minipage}{2cm}
\begin{tikzpicture} [thick]
\draw (0,0) -- (-.35,.6062);
\draw[densely dotted] (-.35,.6062)--(-.5,.866);
\draw (0,0) -- (-.35,-.6062);
\draw[densely dotted] (-.35,-.6062)-- (-.5,-.866);
\draw (0,0) -- (1,0);
\draw (1,0) -- (1.35,.6062);
\draw[densely dotted] (1.35,.6062)--(1.5,.866);
\draw (1,0) -- (1.35,-.6062);
\draw[densely dotted] (1.35,-.6062)--(1.5,-.866);
\end{tikzpicture}
\end{minipage}
$-$
\begin{minipage}{1.732cm}
\begin{tikzpicture} [thick]
\draw (0,0) -- (.4,.231);
\draw (.6,.346) -- (1.35,.78);
\draw[densely dotted] (1.35,.78) -- (1.5,.866);
\draw (0,0) -- (-.35,-.6062);
\draw[densely dotted] (-.35,-.6062)-- (-.5,-.866);
\draw (0,0) -- (1,0);
\draw (1,0) -- (-.35,.78);
\draw[densely dotted] (-.35,.78)--(-.5,.866);
\draw (1,0) -- (1.35,-.6062);
\draw[densely dotted] (1.35,-.6062)--(1.5,-.866);
\end{tikzpicture}
\end{minipage}
\end{center}
\caption{A tree in $\mathcal T_3$, and an IHX relation.}
\end{figure}

Let ${\sf L}_n={\sf L}_n(m)$ be the degree $n$ part of the free $\mathbb Z$-Lie algebra with degree $1$ basis $X_1,\ldots,X_m$. It is spanned by formal non-associative bracketing expressions of $n$ basis elements, modulo the Jacobi identity and self-annihilation relation $[Y,Y]=0$. Replacing this self-annihilation relation with the antisymmetry relation $[Y,Z]=-[Z,Y]$, one obtains ${\sf L}'_n={\sf L}'_n(m)$, the degree $n$ part of the free $\mathbb Z$ \emph{quasi-Lie algebra} on these same generators.

Let ${\sf D}_n$ be the kernel of the bracketing map $\mathbb Z^m\otimes{\sf L}_{n+1}\to{\sf L}_{n+2}$, defined by $X_i\otimes Y\mapsto[X_i,Y]$, where $\sL_1$ is identified with $\Z^m$; and let ${\sf D}'_n$ be the kernel of the corresponding bracketing map, $\mathbb Z^m\otimes{\sf L}'_{n+1}\to{\sf L}'_{n+2}$, of quasi-Lie algebras.

Levine's map $\eta'\colon \mathcal T_n\to{\sf D}'_n$
is defined by the formula
$$\eta'(t)=\sum_v X_{\ell(v)}\otimes B'_v(t)$$
where the sum is over all univalent vertices $v$ of $t$, $\ell(v)$ is the index labeling $v$, and $B'_v(t)$ is defined to be the iterated bracket in ${\sf L}'_{n+1}$ corresponding to the \emph{rooted} oriented tree obtained from $t$ by removing the label of $v$ and letting $v$ be the root. Here is an example.
\begin{align*}
\begin{minipage}{2.5cm}
\begin{tikzpicture} [thick]
\draw (0,0)--(-.5,.866) node [fill=white] {$i$};
\draw (0,0)--(-.5,-.866) node [fill=white] {$l$};
\draw (0,0)--(1,0);
\draw (1,0)--(1.5,.866) node [fill=white] {$j$};
\draw (1,0)--(1.5,-.866) node [fill=white] {$k$};
\end{tikzpicture}
\end{minipage}
\mapsto
&X_l\otimes\left(
\begin{minipage}{1.9cm}
\begin{tikzpicture} [thick]
\fill[black] (0,.0) circle (.05);
\draw (0,0)--(0,.5);
\draw (0,.5)--(-.5,1.2) node [fill=white]{$i$};
\draw (0,.5)--(.5,1.2);
\draw (.5,1.2)--(0,1.9) node [fill=white]{$j$};
\draw (.5,1.2)--(1,1.9) node [fill=white]{$k$};
\end{tikzpicture}
\end{minipage}\right)
+
X_i\otimes\left(
\begin{minipage}{1.9cm}
\begin{tikzpicture} [thick]
\fill[black] (0,.0) circle (.05);
\draw (0,0)--(0,.5);
\draw (0,.5)--(-.5,1.2);
\draw (0,.5)--(.5,1.2) node [fill=white]{$l$};
\draw (-.5,1.2)--(0,1.9) node [fill=white]{$k$};
\draw (-.5,1.2)--(-1,1.9) node [fill=white]{$j$};
\end{tikzpicture}
\end{minipage}\right)\\
&+
X_j\otimes\left(
\begin{minipage}{1.9cm}
\begin{tikzpicture} [thick]
\fill[black] (0,.0) circle (.05);
\draw (0,0)--(0,.5);
\draw (0,.5)--(-.5,1.2) node [fill=white]{$k$};
\draw (0,.5)--(.5,1.2);
\draw (.5,1.2)--(0,1.9) node [fill=white]{$l$};
\draw (.5,1.2)--(1,1.9) node [fill=white]{$i$};
\end{tikzpicture}
\end{minipage}\right)
+
X_k\otimes\left(
\begin{minipage}{1.9cm}
\begin{tikzpicture} [thick]
\fill[black] (0,.0) circle (.05);
\draw (0,0)--(0,.5);
\draw (0,.5)--(-.5,1.2);
\draw (0,.5)--(.5,1.2) node [fill=white]{$j$};
\draw (-.5,1.2)--(0,1.9) node [fill=white]{$i$};
\draw (-.5,1.2)--(-1,1.9) node [fill=white]{$l$};
\end{tikzpicture}
\end{minipage}\right)
\end{align*}

In \cite{L2} Levine conjectured that, for every $n$, the map $\eta'_n\colon \mathcal T_n\to {\sf D}'_n$ is an isomorphism, which has implications concerning the precise relationship between two filtrations of the group of homology cylinders. He obtained partial progress toward the conjecture in the form of the following theorem, and even obtained more progress in \cite{L3}.

\begin{thm}[Levine]\label{levine2}
$\eta'_n\colon\mathcal T_n\to{\sf D}'_n$ is a split surjection. $\Ker \eta'_n$ is the torsion subgroup of $\mathcal T_n$ if $n$ is even. It is the odd torsion subgroup if $n$ is odd. In either case
$$(n+2)\Ker \eta'_n=0.$$
\end{thm}

By \cite{L3}, ${\sf D}'_n$ is torsion-free when $n$ is even, and only has $2$-torsion of the form $X_i\otimes[Z,Z]$ when $n$ is odd. Hence the Levine conjecture boils down to saying that the $\mathcal T_n$ groups have no torsion except the $2$-torsion coming from symmetric trees $\tree{i}{J}{J}$, which are $2$-torsion by the antisymmetry relation. In the theory of Whitney tower intersections, there is a commutative diagram
$$\xymatrix{\mathbb Q\otimes\cT_{n}\ar@{->>}[r]\ar[dr]_\eta&\mathbb Q\otimes{\sf W}_{n}\ar[d]^{\mu_{n}}\\
&\mathbb Q\otimes{\sf D}'_{n}
}$$
where the group ${\sf W}_n$ is defined as the set of links bounding order $n$ Whitney towers, modulo order $(n+1)$ Whitney concordance, and $\mu_n$ is the total Milnor invariant of order $n$ (length $n+2$).
Thus the fact that $\eta$ is a rational isomorphism implies that the total Milnor invariant $\mu_n$ completely classifies the associated graded group $\mathbb Q\otimes{\sf W}_n$. 
So a precise understanding of the torsion in $\mathcal T_n$ is necessary for us to be able to classify the groups ${\sf W}_n$ over the integers. It turns out that the Levine conjecture is crucial in establishing that ${\sf W}_n$ is completely classified by Milnor invariants, higher-order Sato-Levine invariants and  higher-order Arf invariants. See \cite{CST0,CST1,CST2,CST4} for details.

In our proof of the Levine conjecture, the first step is to reinterpret ${\sf L}'_n$ as the zeroth homology of a chain complex $\mathbb L_{\bullet,n}$ of rooted labeled trees, where internal vertices may have valence higher than three, with the homological degree given by excess valence. The bracketing operation  lifts to an injective map on the chain complex level, giving us a short exact sequence
$$0\to \mathbb Z^m\otimes\mathbb L_{\bullet,n+1}\to \mathbb L_{\bullet,n+2}\to\overline{\mathbb L}_{\bullet,n+2}\to 0$$
where the complex on the right is by definition the cokernel. Thus we obtain the long exact sequence 
$$H_1(\mathbb L_{\bullet,n+2})\to H_1(\overline{\mathbb L}_{\bullet,n+2})\to\mathbb Z^m\otimes{\sf L}'_{n+1}\to{\sf L}'_{n+2}\to H_0(\overline{\mathbb L}_{\bullet,n+2})\to 0.$$
Because the bracketing map is onto, $ H_0(\overline{\mathbb L}_{\bullet,n+2})=0$. Thus, if
 $H_1(\mathbb L_{\bullet,n+2})$ were equal to zero, then we would have an alternate characterization of ${\sf D}'_n$ as $H_1(\overline{\mathbb L}_{\bullet,n+2})$. Indeed, in Section~\ref{sec:vanish-thm-proof} we show that this homology does vanish:
 \begin{thm}\label{thm:vanish}
$H_1(\mathbb L_{\bullet,n};\mathbb Z)=0$.
\end{thm}
With ${\sf D}'_n\cong H_1(\overline{\mathbb L}_{\bullet,n+2})$,
the map $\eta'$ turns into a map $\bar\eta\colon\mathcal T_n\to H_1(\overline{\mathbb L}_{\bullet,n+2})$ which sums over adding a rooted edge to every internal vertex of tree $t\in\mathcal T_n$.

   To complete the proof of Theorem~\ref{mainthm}, in section \ref{sec:proof} we show that $H_1(\overline{\mathbb L}_{\bullet,n+2})$ is isomorphic to $\cT_n$, via a chain map $\bbeta_\bullet$ which is a lift of $\bar\eta$.
\begin{thm}\label{lem:reducedT}
$\bbeta_\bullet$ induces an isomorphism
$\bar\eta\colon{\cT}_n\overset{\cong}{\longrightarrow} H_1(\overline{\mathbb L}_{\bullet,n+2}).$ 
\end{thm}

The proofs of Theorems~\ref{thm:vanish} and \ref{lem:reducedT} use the powerful technique of discrete Morse theory for chain complexes, which we discuss in section ~\ref{sec:discretemorse}.  Roughly, a discrete gradient vector field on a chain complex is a list of pairs of generators each giving combinatorial data parameterizing an atomic acyclic subcomplex.  These acyclic subcomplexes can then be modded out to obtain a simpler quasi-isomorphic complex, called the Morse complex. In practice, it is often possible to find a discrete vector field that drastically reduces the size of the complex being studied.  Indeed, in order to show $H_1({\mathbb L}_{\bullet,n+2})=0$, we construct a discrete vector field for which the Morse complex is $0$ in degree $1$! This vector field is
inspired by the Hall basis algorithm for the free Lie algebra. Its lowest degree vectors are defined directly from this algorithm, with a suitably nice choice of a Hall order on trees. One of the conditions on a gradient vector field is that there are no ``gradient loops," which in practice is often the trickiest thing to verify. In this case, the fact that the Hall basis algorithm ``works" allows us to rule out loops involving these lowest degree vectors. The complete vector field is a natural extension of the lowest degree case, and ruling out gradient loops in general involves an exhaustive case analysis. 

 Actually the previous paragraph simplifies the real story somewhat, in that discrete Morse theory works well for \emph{free} chain complexes, but the chain groups ${\mathbb L}_{\bullet,n+2}$ have both $\Z$ and $\z$ direct summands. To get around this, we actually construct two slightly different vector fields, one for  $\zh\otimes\mathbb L_{\bullet,n+2}$ and one for $\z\otimes\mathbb L_{\bullet,n+2}$. The former complex kills symmetric trees of degree $0$, so that $H_0(\mathbb L_{\bullet,n+2};\zh)\cong\zh\otimes{\sf L}_n$ is the free Lie algebra over $\zh$. The second complex is more closely tied to the free quasi-Lie algebra, and indeed the vector field we construct for the  case of $\z$-coefficients is closely aligned with Levine's generalization of the Hall basis algorithm to the quasi-Lie case \cite{L2}.  

Now we will discuss the proof of Theorem~\ref{lem:reducedT}. The first step is to generalize $\eta'$ to a chain map $${\bbeta}_\bullet\colon \mathbb T_\bullet\to\overline{\mathbb L}_{\bullet+1},$$
where $\mathbb T_\bullet$ is a chain complex of \emph{unrooted} trees whose zeroth homology is $\mathcal T_n$. Recall that the chain complex $\overline{\mathbb L}_{\bullet,n+2}$ is defined as a quotient of an abelian group of rooted trees by the image of the bracketing map.
We construct a discrete vector field $\Delta$ on $\overline{\mathbb L}_{\bullet,n+2}$, essentially by picking a basepoint and pushing the root away from it, when possible. (This is subtle because trees have nontrivial automorphisms, so one has to be careful doing this.) In any event,
this gives rise to the Morse complex 
$\overline{\mathbb L}^\Delta_\bullet$, so that the composition
$$ \mathbb T_\bullet\to\overline{\mathbb L}_{\bullet+1}\to \overline{\mathbb L}^\Delta_{\bullet+1},$$
has kernel and cokernel which are  easy to analyze. In particular, $\bbeta_\bullet$ induces an isomorphism of homologies in degree $0$, though not in higher degrees. 

The \emph{signature} of a tree is an $m$-tuplet $\sigma=(n_1,\ldots,n_m)$ that records the multiplicities of each label $1,...,m$. Definining $|\sigma|=n_1+\cdots+n_m$,
$\mathcal T_n=\oplus_{|\sigma|=n+2} \mathcal T_\sigma$, where $\cT_\sigma$ is the subgroup of $\mathcal T_n$ spanned by trees with signature $\omega$.
 For the reader's amusement, in Figure~\ref{tautable}, we list computer calculations of the $\cT_\sigma$ groups for small values of $\sigma=(j,k)$.
 
\begin{figure}
{
$
\begin{array}{c|llllllll}
&2&3&4&5&6&7&8&9\\
\hline
2&\Z&{ \z}&\Z&{ \z}&\Z              &{ \z}&\Z              &{ \z}\\
3&{\z}&\Z&{ \z}&\Z&\Z\oplus{ \z}&\Z&\Z\oplus{ \z}&\Z^2\\
4&\Z&{ \z}&\Z^2&\Z\oplus{ \z}&\Z^3&\Z^2\oplus{ \z^2}&\Z^5&\Z^3\oplus{ \z^2}\\
5&{\z}&\Z&\Z\oplus{ \z}&\Z^3&\Z^3\oplus{ \z^2}&\Z^6&\Z^7\oplus{ \z^2}&\Z^{11}\\
6&\Z&\Z\oplus{\z}&\Z^3&\Z^3\oplus{\z^2}&\Z^9&\Z^9\oplus{\z^3}&\Z^{19}&\Z^{22}\oplus{\z^5}\\
7&{\z}&\Z&\Z^2\oplus{\z^2}&\Z^6&\Z^9\oplus{\z^3}&\Z^{19}&\Z^{28}\oplus{\z^5}&{\Z^{47}}\\
8&\Z&\Z\oplus{\z}&\Z^5&\Z^7\oplus{\z^2}&\Z^{19}&\Z^{28}\oplus{\z^5}&{\Z^{58}}&\\
9&{\z}&\Z^2&\Z^3\oplus{\z^2}&\Z^{11}&\Z^{22}\oplus{\z^5}&{\Z^{47}}&&\\
10&\Z&\Z\oplus{\z}&\Z^7&\Z^{13}\oplus{\z^3}&{\Z^{36}}&&&\\
11&{\z}&\Z^2&\Z^5\oplus{\z^3}&{\Z^{18}}&&&&\\
12&\Z&\Z^2\oplus{\z}&\Z^9&&&&&
\end{array}
$}
\caption{A computer generated table of the groups $\mathcal T_{(j,k)}$.}\label{tautable}
\end{figure}

{\bf Acknowledgements:} This paper was written while the first two authors were visiting the third author at the Max-Planck-Institut f\"ur Mathematik in Bonn. They all thank MPIM for its stimulating research environment and generous support. The first author was partially supported by NSF grant  DMS-0604351 and the last author was also supported by NSF grants DMS-0806052 and DMS-0757312.
 We thank Daniel Moskovich for comments on an early draft.


\section{Tree homology}\label{sec:Levine}
In this section, we interpret Levine's conjecture in a homological setting.
It is well-known that the free (quasi)-Lie algebra can be regarded as the zeroth
homology of a complex of rooted oriented trees (of arbitrary
valence $\neq 2$), with univalent vertices labeled by the
generators, since the boundary of a tree with a $4$-valent vertex
is precisely a Jacobi relator. Over the rationals, all the
homology is concentrated in degree zero (Proposition~\ref{nohom}), but the integral homology appears to be unknown.

\begin{defn}
Throughout this paper all trees are allowed to have vertices of
any valence other than 2, and are considered up to isomorphism. 

 An \emph{orientation} of a tree is an
equivalence class of orderings of the edges, where two orderings are
equivalent if they differ by an even permutation. Each tree has
at most two orientations, and one is said to be the negative of
the other. Note that although the orientation given here seems different from
the standard orientation of unitrivalent trees or graphs given by
ordering the half-edges around a vertex, according to \cite{CV}
Proposition 2, these are equivalent notions for odd-valent trees.

Following the previous notation for unitrivalent trees, labels from
the index set $\{1,2,\ldots,m\}$ are used to decorate univalent
vertices, and a \emph{rooted tree} has all univalent vertices labeled
except for a single un-labeled \emph{root} univalent vertex. All non-root
univalent vertices are called \emph{leaves}, and all vertices of
valence $\geq 3$ all called \emph{internal vertices}.

The \emph{bracket} of two oriented rooted trees is the rooted tree
$(J_1,J_2)$ defined by identifying the roots of $J_1$ and $J_2$ and
attaching an edge to the identified vertex, the other vertex of this
edge being the new root. The orientation is given by ordering the
root edge first, then the edges of $J_1$ in the order prescribed by
$J_1$'s orientation, and finally the edges of $J_2$ in the order
prescribed by its orientation.

The \emph{homological degree} of a tree is defined to be $\sum_v (|v|-3)$
where the sum is over all internal vertices $v$, and $|v|$ represents the valence
of the vertex.

\end{defn}

Rooted trees will usually be denoted by capital letters, and unrooted trees by lower case letters.

\begin{defn}
In the following chain complexes, we divide by the relation
$(T,-\mathrm{or})=-(T,\mathrm{or})$ for every oriented tree $(T,\mathrm{or})$. 

For all $k\geq 2$, let $\mathbb L_{\bullet,k}$ be the chain complex spanned by oriented rooted
trees with $k+1$ total univalent vertices, and the $k$
leaves labeled by the cardinality $m$ index set. The trees are
graded by homological degree. 

Let $v$ be an internal vertex of a tree $J$ of valence $\geq 4$, and let $P$ be an unordered partition of the half-edges incident to $v$ into two sets each having at least two elements.  $P$ determines an \emph{expansion} of $J$, where the vertex $v$ expands into a new edge $e$, and the half-edges are distributed to the ends of $e$ according to the partition $P$. The induced orientation of an expansion is defined by numbering the new edge first, and increasing the numbering of the other edges by one. 
 
The
boundary operator $\partial\colon \mathbb L_{\bullet,k}\to \mathbb
L_{\bullet-1,k}$ is defined by setting $\partial J$ equal to the sum of all expansions of $J$. See Figure~\ref{ihxfig}, which shows the three expansions of a $4$-valent vertex. Note that $\partial$ vanishes on degree zero trees since they have only trivalent internal vertices.

\begin{figure}
\begin{center}
$\begin{minipage}{.7in}
\includegraphics[width=.7in]{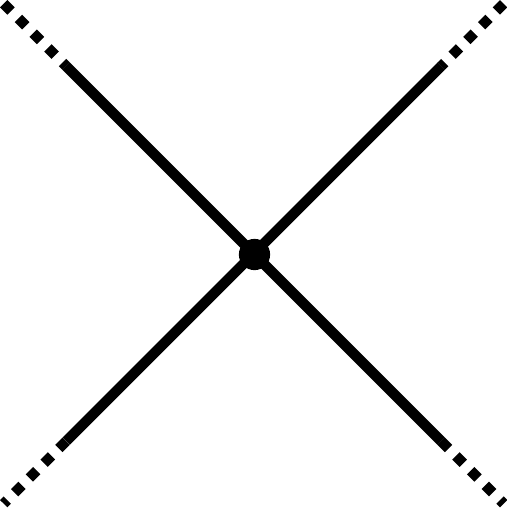}
\end{minipage}\overset{\partial}\mapsto
\begin{minipage}{.7in}
\includegraphics[width=.7in]{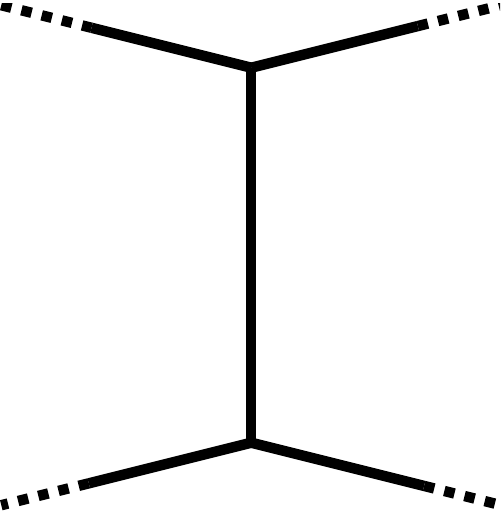}
\end{minipage}+
\begin{minipage}{.7in}
\includegraphics[width=.7in]{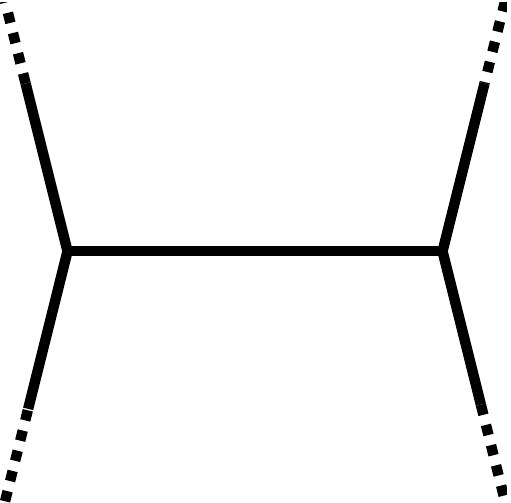}
\end{minipage}+
\begin{minipage}{.7in}
\includegraphics[width=.7in]{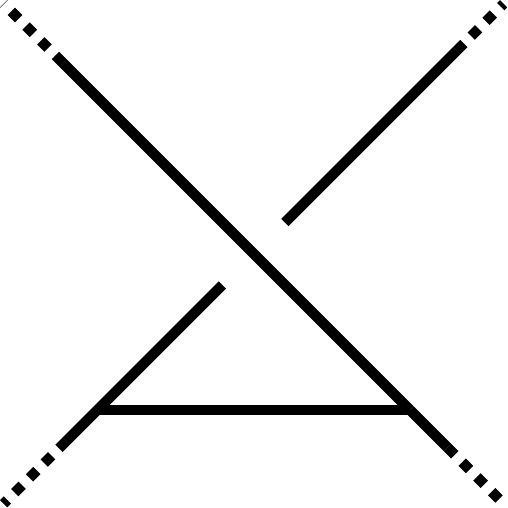}
\end{minipage}
$
\end{center}
\caption{An IHX relation appearing as the image of $\partial$.}\label{ihxfig}
\end{figure}

The chain complex $\mathbb T_{\bullet,k}$ is spanned by oriented labeled
\emph{un}rooted trees with $k$ leaves. The trees are again graded by homological degree. The boundary operator $\partial\colon\mathbb T_{\bullet,k}\to\mathbb T_{\bullet-1,k}$ is defined as before by setting $\partial t$ equal to the sum of all expansions of $t$.

The chain complex $\overline{\mathbb L}_{\bullet,k}$ is the quotient complex of ${\mathbb L}_{\bullet,k}$
by the subcomplex spanned by trees of the form $(i,J)=\tree{}{i}{J}$. Here the notation $i$ stands for the rooted tree having a single 
$i$-labeled leaf.
\end{defn}

Some of the homology groups of these chain complexes turn out to be relevant to us.

\begin{prop}\label{prop:isomorphisms}
We have the following isomorphisms ($\Z$-coefficients):
\begin{enumerate}
\item $H_{0}(\mathbb L_{\bullet,n})\cong {\sf L}'_n$
\item $H_{0}(\mathbb T_{\bullet,n+2})\cong\mathcal T_n$
\item $H_1(\overline{\mathbb L}_{\bullet,n+2})\cong{\sf D}'_n$ for $n>0$.
\end{enumerate}
\end{prop}
\begin{proof}
The first isomorphism comes from the fact that in homological degree $0$ all
trees are trivalent, and hence they are all cycles. The image of
the boundary operator is precisely the submodule of IHX relators,
since $\partial$ expands a $4$-valent vertex into an IHX relator (Figure~\ref{ihxfig}),
where the signs are verified in \cite{CV} p.1207. The second
isomorphism is similar. 

We proceed to explain the third isomorphism, which depends on Theorem~\ref{thm:vanish}.
Recall that by definition ${\sf D}'_n$ is the kernel of the bracketing map 
$$
\Z^m\otimes{\sf L}'_{n+1}\to {\sf L}'_{n+2}
$$
via the identification $\sL'_1\cong\Z^m$.This bracketing operation on the quasi-Lie algebra lifts to a chain map 
$$
\br_\bullet\colon \mathbb Z^m\otimes \mathbb L_{\bullet,n+1}\to \mathbb L_{\bullet,n+2}
$$ 
which sends $X_i\otimes J$ to the oriented tree $(i,J)$.  Note that for $n>0$, $\br_\bullet$ is \emph{injective} at the chain level. Thus we get a short exact sequence of chain complexes:
$$
0\to \Z^m\otimes \mathbb L_{\bullet,n+1}\to\mathbb L_{\bullet,n+2}\to \overline{\mathbb L}_{\bullet,n+2}\to 0
$$
where $\overline{\mathbb L}_{\bullet,n+2}$ is by definition the cokernel of $\br_\bullet$. Using statement~(i) of the proposition, this gives rise to the long exact sequence:
$$H_1(\mathbb L_{\bullet,n+2})\to H_1(\overline{\mathbb L}_{\bullet,n+2})\to\mathbb Z^m\otimes{\sf L}'_{n+1}\to{\sf L}'_{n+2}\twoheadrightarrow H_0(\overline{\mathbb L}_{\bullet,n+2})$$
 We will prove later (Theorem~\ref{thm:vanish})  that $H_1(\mathbb L_{\bullet,n+2})=0$. Since the bracketing map is onto, we get the short exact sequence
$$0\to H_1(\overline{\mathbb L}_{\bullet,n+2})\overset{\kappa}{\to}\mathbb Z^m\otimes{\sf L}'_{n+1}\to{\sf L}'_{n+2}\to 0$$
where $\kappa$ is the connecting homomorphism from the long exact sequence.
Hence ${\sf D}'_n\cong H_1(\overline{\mathbb L}_{\bullet,n+2})$.
\end{proof}

Let us interpret $\eta'_n$ in this context. Clearly $\eta'_n$ lifts uniquely to a map $\bar\eta_n$ as in the diagram below:
$$
\xymatrix{
\cT_n\ar@{-->}[d]_{\bar\eta_n}\ar[dr]^{\eta'_n}&\\
H_1(\overline{\mathbb L}_{\bullet,n+2})\ar@{>->}[r]^{\kappa}&\Z^m\otimes {\sf L}'_{n+1}\ar@{->>}[r]&{\sf L}'_{n+2}
}
$$
Suppose $t\in\mathbb T_{0,n+2}$ is an oriented tree.
Define $t^r\in\mathbb L_{1,n+2}$ to be the sum of adding a root edge, numbered first in the orientation, to all of the internal vertices of $t$:

$$\underset{t}{\begin{minipage}{1in}\includegraphics[width=1in]{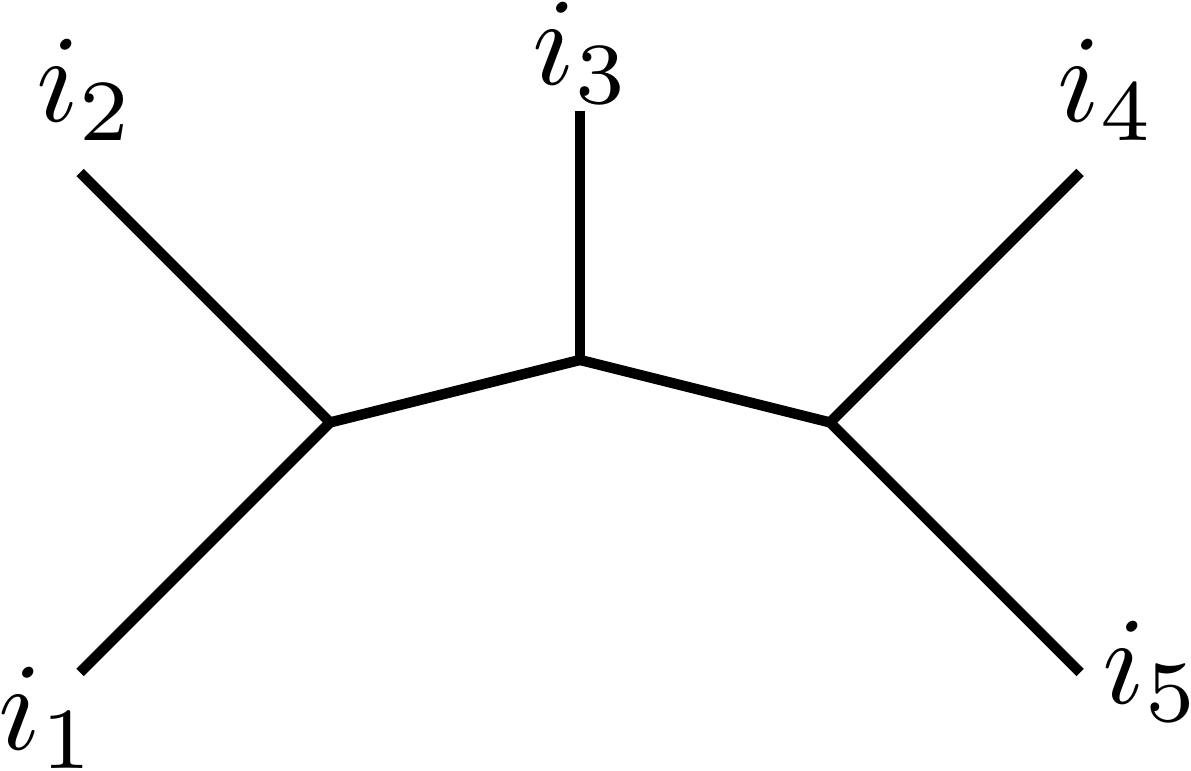}\end{minipage}}\mapsto\underset{t^r}{\begin{minipage}{1in}\includegraphics[width=1in]{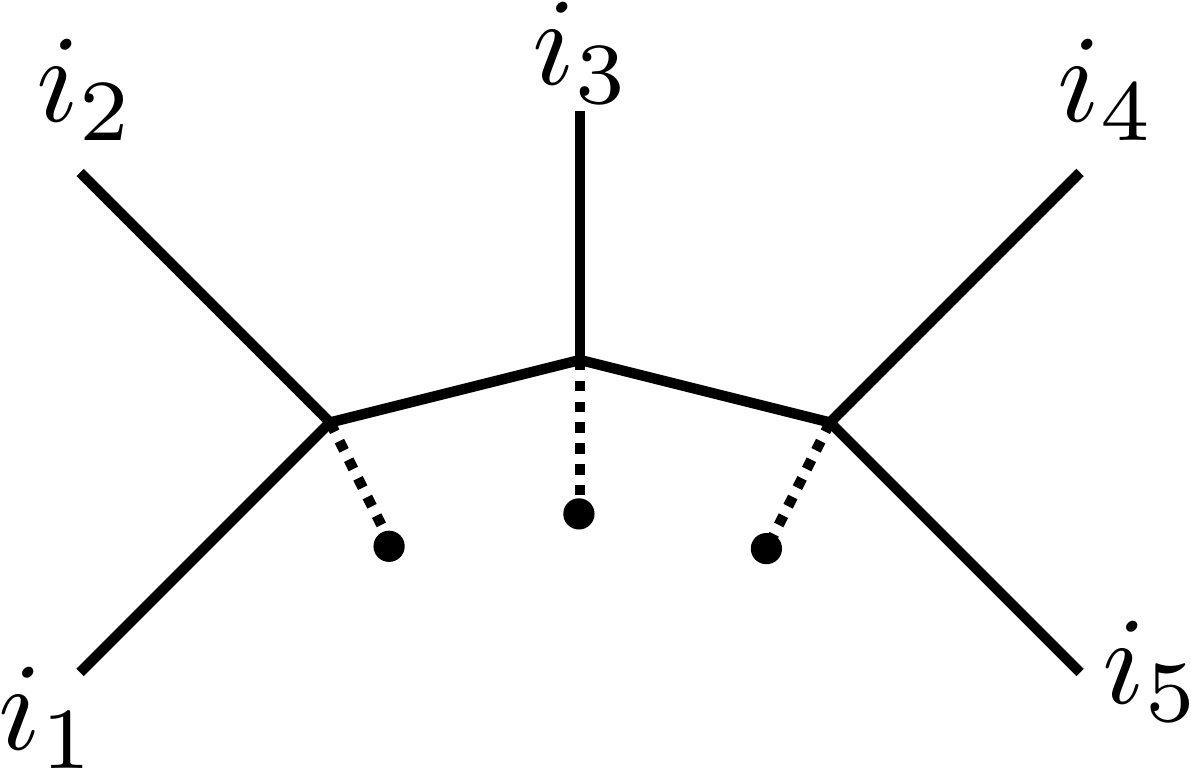}\end{minipage}}$$

The dotted edges mean that we are summing over putting the root edge in each position.
 We claim that $\bar \eta_n(t)=t^r$. So we must verify that $\kappa(t^r)=\eta_n'(t)$. The map $\kappa$ is defined via the snake lemma as in the diagram below.
$$ 
 \xymatrixrowsep{0pc}\xymatrix{
 \eta'_n(t)\ar@{|->}[r]&\sum_v(\ell_v(t),B'_v(t))&\\
 \Z^m\otimes\mathbb L_{0,n+1}\ar@{>->}[r]&\mathbb L_{0,n+2}\ar@{->>}[r]&\overline{\mathbb L}_{0,n+2}\\
 &&\\
  &&\\
   &&\\
 \Z^m\otimes\mathbb L_{1,n+1}\ar[uuuu]^{1\otimes \partial}\ar@{>->}[r]& \mathbb L_{1,n+2}\ar[uuuu]^{\partial}\ar@{->>}[r]&\overline{\mathbb L}_{1,n+2}\ar[uuuu]^{\bar\partial}\\
 &t^r\ar@{|->}@/^5pc/[uuuuuu]\ar@{|->}[r]&t^r
 }
$$
Here we use that $\partial t^r=\sum_v(\ell_v(t),B'_v(t))$ because of internal cancellation of the root:
$$
\partial (t^r)=\begin{minipage}{1.2in}\includegraphics[width=1.2in]{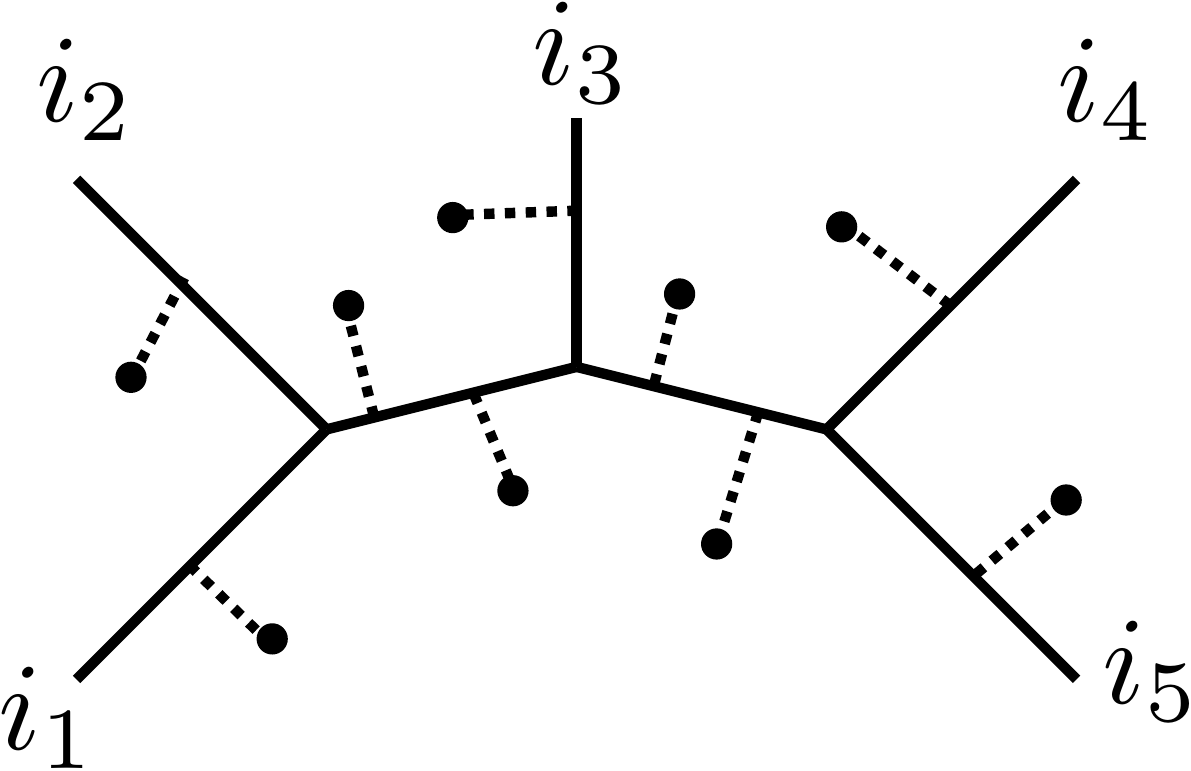}\end{minipage}=
\begin{minipage}{1.2in}\includegraphics[width=1.2in]{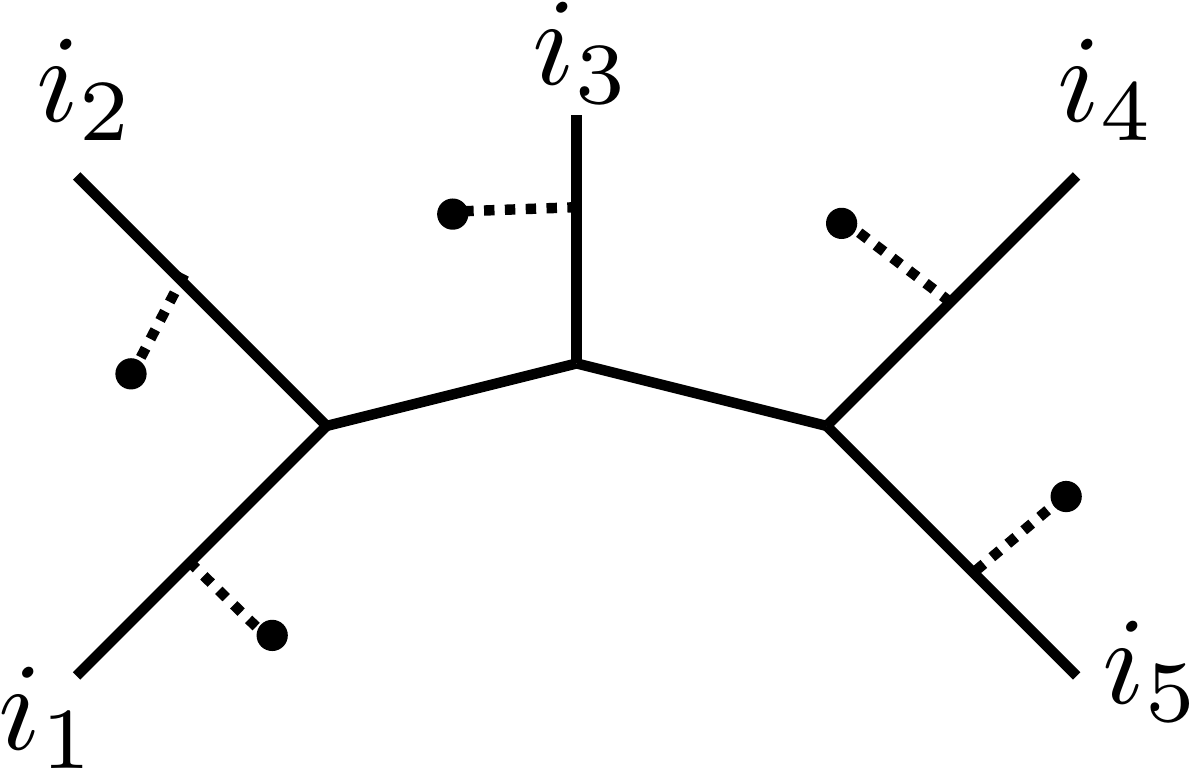}\end{minipage}=\sum_v(\ell_v(t),B'_v(t))
$$
Tracing through the diagram indeed shows that $t^r\mapsto\eta'_n(t)$.
 
In fact, $\bar\eta$ extends to a map of chain complexes, a fact which we will need in section~\ref{sec:proof}. Define $\bbeta_\bullet\colon \mathbb T_{\bullet,n}\to\overline{\mathbb L}_{\bullet+1,n}$ by letting $\bbeta_\bullet(t)$ be the sum over attaching a root edge to every internal vertex of $t$, mutiplied by the sign $(-1)^{\operatorname{deg}(t)}$.

\begin{lem}\label{lem:eta-chain-map}
$\bbeta_\bullet\colon \mathbb T_{\bullet,n}\to\overline{\mathbb L}_{\bullet+1,n}$ is a chain map.
\end{lem}
 \begin{proof}
 For any internal 
vertex $v$ of an unrooted tree $t$, let $\alpha_v(t)$ denote the rooted tree gotten by attaching a root edge to $v$, so that 
$\bbeta_\bullet(t)=\sum_v \alpha_v(t)$.
By definition $\partial\bbeta_\bullet(t)$ is a sum of expansions of $\bbeta_\bullet(t)$, and since those expansions which push the root edge onto an interior edge of $t$ all cancel in pairs, the only terms that are relevant are those where the underlying tree $t$ is expanded. Fix such an expansion, where the vertex $v_0$ of $t$ gets expanded into two vertices  $v_1$ and $v_2$, connected by an edge. Call this expanded tree $t^e$. Then if $v$ is a vertex of $t$ that is not $v_0$, $\partial\alpha_v(t)$ contains one term, $\alpha_v(t^e)$, corresponding to the fixed expansion $t^e$. On the other hand $\partial\alpha_{v_0}(t)$ contains two such terms: $\alpha_{v_1}(t^e)$ and $\alpha_{v_2}(t^e)$. So for every internal vertex of the expanded tree $t^e$, there is exactly one summand where the root attaches to it. 
 Thus $\partial\bbeta_\bullet(t)=\sum\bbeta_\bullet(t^e)$ where the sum is over all expansions of $t$, and so by definition 
 $\partial\bbeta_\bullet(t)=\bbeta_\bullet\partial(t)$. The extra factor of $(-1)^{\operatorname{deg}(t)}$ in the definition of $\bbeta$ is designed to make the orientations in this equation agree.
  \end{proof}

\subsection{On the rational homology}
Although not necessary for the main results of this paper, the following proposition confirms that rationally all the homology is concentrated in degree $0$, and gives a crude estimate for the torsion.
\begin{prop} \label{nohom}
\item
\begin{enumerate}
\item $H_{k}(\mathbb L_{\bullet,n};\mathbb Z)$ is $n!$-torsion, and so $H_{k}(\mathbb L_{\bullet,n};\mathbb Q)=0$,  for all $k\geq 1$.
\item $H_{k}(\mathbb T_{\bullet,n+2};\mathbb Z)$ is $(n+2)!$-torsion, and so $H_{k}(\mathbb T_{\bullet,n+2};\mathbb Q)=0$, for all $k\geq 1$.
\end{enumerate}
\end{prop}
\begin{proof}
Consider the tree complex $\mathbb L^\dagger_{\bullet,n}$ defined
analogously to $\mathbb L_{\bullet,n}$ except that the leaves are always
labeled by $1,\ldots, n$ without repeats. This actually corresponds
to an augmented cochain complex for a simplicial complex $K_n$
defined in the following way. Every tree, except the unique one with
only one internal vertex, corresponds to a  simplex, given by putting nonnegative
lengths summing to $1$ on all of its internal edges. (When an
edge has length $0$, it contracts to a point.) So we have the
isomorphism $\widetilde{H}^i(K_n)\cong H_{n-i-3}(\mathbb
L^\dagger_{\bullet,n})$, where the index shift comes from the fact that
the dimension, $i$, of a simplex is one less than the number of internal edges
of the corresponding tree, and the reader may verify that this corresponds to homological
degree $n-i-3$.
This simplicial complex
$K_n$, also known as the Whitehouse complex, is well known to
be homotopy equivalent to a wedge of $(n-1)!$ spheres of dimension
$n-3$ (which is the top dimension). See \cite{Re} for an elementary
proof. In particular, this implies that
$H_k(\mathbb L^\dagger_{\bullet,n})=0$ for all $k\geq 1$.  
Let $V=\mathbb Z^m$ be the abelian group spanned by the label set.
The symmetric group $\Sigma_n$ acts on
$\mathbb L^\dagger_{\bullet,n}$ by permuting the labels, and we have an
isomorphism with the space of coinvariants 
$$\mathbb
L_{\bullet,n}=[V^{\otimes n}\otimes \mathbb L^\dagger_{\bullet,n}]_{\Sigma_n}$$
where $\Sigma_n$ acts simultaneously on $L^\dagger_{\bullet,n}$ and
$V^{\otimes n}$. Now, if $k>0$, $H_i(V^{\otimes n}\otimes \mathbb
L^\dagger_{\bullet,n})\cong V^{\otimes n}\otimes H_i( \mathbb
L^\dagger_{\bullet,n})=0$. The proof is finished by noting that if a finite
group $G$ acts on a chain complex $C_\bullet$, where $H_i(C_\bullet)=0$, then
$H_i([C_\bullet]_G)$ is $|G|$-torsion. To see this note that we have a
sequence
$$ [C_\bullet]_G\to C_\bullet\to [C_\bullet]_G$$ where the first map is the map $\sigma\mapsto \sum_{g\in G} g\cdot\sigma$,
and the second map is the natural quotient. Their composition is $|G|\cdot\operatorname{Id}$.
Applying the functor $H_i(\cdot)$, we have $|G|\cdot\operatorname{Id}$ factoring through $0$,
implying that $H_i([C_\bullet]_G)$ is $|G|$-torsion.

The proof for $\mathbb T_{\bullet,n+2}$ is similar, but one needs to mod out
by the action of $\Sigma_{n+2}$, and thereby include the root,
instead.
 \end{proof}

\section{Discrete Morse Theory for Chain Complexes}\label{sec:discretemorse}
In order to prove Theorems~\ref{mainthm} and \ref{thm:vanish},
we adopt a convenient general framework for constructing quasi-isomorphisms based on Forman's theory of discrete vector fields on simplicial complexes \cite{For}. In the setting we require, this has been studied by Kozlov \cite{Kozlov}, who proves that the Morse complex, defined below, yields isomorphic homology, but does not construct a map to the Morse complex. We give an elaboration of his proof which has the added benefit of constructing a map to the Morse complex, but we claim no originality.

We start by considering, like Kozlov, chain groups which are free modules over a commutative ring. We then analyze a specific non-free case: when the chain groups are direct sums of copies of $\Z$ and $\z$. This analysis can be generalized to other non-free cases, but we limit ourselves to what we need in this paper.

\begin{defn}[Homological Vector Field]
Fix a commutative ground ring $R$, with unit, and suppose that $(C_\bullet,\partial)$ is a chain complex where each $C_k$ is a free $R$-module, with a fixed basis $\{\mathbf {b_k^i}\}$.
\begin{enumerate}
\item  A \emph{vector} is a pair of basis elements $(\mathbf {b^i_{k-1}},\mathbf {b_k})$ in degrees $k-1$ and $k$ respectively, such that 
 $\partial(\mathbf {b_{k}})=r_i\mathbf{ b_{k-1}^i}+\sum_{i\neq j} c_j\mathbf {b_{k-1}^j}$, where $r_i\in R$ is invertible, and the coefficients $c_j\in R$ are arbitrary.
  \item A \emph{(homological) vector field}, $\Delta$, is a collection of vectors $(\mathbf a,\mathbf b)$ such that every basis element appears in at most one vector of $\Delta$. 
  \item A basis element is said to be \emph{critical} if doesn't appear in any vector of the vector field $\Delta$.  The set of all critical basis elements for $\Delta$ will be denoted $\mathfrak X^\Delta$. 
 \item Given a vector field, $\Delta$,
a \emph{gradient path} is a sequence of basis elements $$\mathbf {a_1}, \mathbf {b_1}, \mathbf {a_2},\mathbf {b_2},\ldots, \mathbf{a_m}$$ where 
each $(\mathbf {a_i},\mathbf {b_i})\in\Delta$, and
 $\mathbf {a_i}$ has nonzero coefficient in $\partial \mathbf {b_{i-1}}$ and $\mathbf{a_i}\neq\mathbf {a_{i-1}}$. It is often useful to visualize gradient paths using a ``zigzag" diagram like the one below.
 $$
\xymatrixrowsep{1pc}
\xymatrixcolsep{1pc}
\xymatrix{
&\mathbf{b_1}\ar[dr]^\partial&&\mathbf{b_2}\ar[dr]^\partial&&\mathbf{b_{m-1}}\ar[dr]^\partial\\
\mathbf{a_1}\ar[ur]^\Delta&&\mathbf{a_2}\ar[ur]^\Delta&&\mathbf{a_3}\ar@{..}[ur]&&\mathbf{a_m}
}
$$
 The set of all gradient paths from $\mathbf a$ to $\mathbf a'$ (that is with $\mathbf {a_1}= \mathbf a$ and $\mathbf {a_m}= \mathbf a'$) is denoted $\Gamma(\mathbf a,\mathbf a')=\Gamma^\Delta(\mathbf a, \mathbf a')$.
 \item A $\partial$-gradient path is a sequence of basis elements $$\mathbf{b_0},\mathbf{a_1},\mathbf{b_1},\ldots,\mathbf{a_m}$$ where $\mathbf{a_1},\mathbf{b_1},\cdots,\mathbf{a_m}$ is a gradient path and $\mathbf{a_1}$ has nonzero coefficient in $\partial \mathbf{b_0}$. The appropriate zigzag here is
  $$
\xymatrixrowsep{1pc}
\xymatrixcolsep{1pc}
\xymatrix{
\mathbf{b_0}\ar[dr]^\partial&&\mathbf{b_1}\ar[dr]^\partial&&\mathbf{b_2}\ar[dr]^\partial&&\mathbf{b_{m-1}}\ar[dr]^\partial\\
&\mathbf{a_1}\ar[ur]^\Delta&&\mathbf{a_2}\ar[ur]^\Delta&&\mathbf{a_3}\ar@{..}[ur]&&\mathbf{a_m}
}
$$
 The set of all $\partial$-gradient paths from $\mathbf b$ to $\mathbf a$ will be denoted $\Gamma_\partial(\mathbf b,\mathbf a)=\Gamma_\partial^\Delta(\mathbf b,\mathbf a)$.
  \item A vector field is said to be a \emph{gradient field} if there are no closed gradient paths. 
 \end{enumerate}
 \end{defn}

Given a vector field, one can construct a degree $1$ homomorphism of the same name $\Delta\colon C_\bullet\to C_{\bullet+1}$ as follows. If $\mathbf a$ is a basis element appearing in a vector $(\mathbf a,\mathbf b)$, define $\Delta(\mathbf a)=\mathbf b$, and define $\Delta$ to be zero on all other basis elements. 

There are functionals on gradient paths and $\partial$-gradient paths
$$
w\colon \Gamma(\mathbf a,\mathbf a')\to R\text{ and }w\colon \Gamma_\partial(\mathbf b,\mathbf a)\to R
$$
called the \emph{weight}.
The weight of a gradient path $\gamma=(\mathbf {a_1}, \mathbf {b_1}, \mathbf {a_2},\mathbf {b_2},\ldots, \mathbf{a_m})$ is defined as follows. For each $\mathbf{a_i}$, where $i>1$, suppose that
 $\partial \mathbf{b_{i-1}}=r_{i-1}\mathbf{a_{i-1}}+c_{i-1} \mathbf {a_i}+\cdots,$ for $r_{i-1},c_{i-1}\in R$, with $r_{i-1}$ invertible.
  Define the weight to be 
  $$
  w_\gamma= (-1)^{m-1} \frac{c_1\cdots c_{m-1}}{r_1\cdots r_{m-1}}\in R.
  $$ 
  The weight of a $\partial$-gradient path $\mu$ is also multiplied by the coefficient $c_0$ of $\mathbf{a_1}$ in $\partial \mathbf{b_0}$: 
$$
w_\mu=(-1)^{m-1}\frac{c_0c_1\cdots c_{m-1}}{r_1\cdots r_{m-1}}\in R.
$$

We now define the Morse complex $C_\bullet^{\Delta}$ for a gradient vector field $\Delta$. The chain groups of $C_\bullet^{\Delta}$ are the submodules of $C_\bullet$ spanned by critical basis elements.  The boundary operator $\partial^{\Delta}$ is defined as follows. Suppose $\mathbf b\in\mathfrak X^\Delta$.
$$\partial^{\Delta}(\mathbf b)=\sum_{\mathbf a\in\mathfrak X^\Delta} d_{ba} \mathbf a$$
where 
$$d_{ba}=\sum_{\gamma\in \Gamma_{\partial}(\mathbf b,\mathbf a)} w_\gamma.$$

There is a map $\phi=\phi^\Delta\colon(C_\bullet,\partial)\to(C_\bullet^{\Delta},\partial^{\Delta})$ defined as follows. If $\mathbf a$ is critical, then $\phi(\mathbf a)=\mathbf a$. Otherwise 
$$
\phi (\mathbf a)=\sum_{\mathbf a'\in\mathfrak X^\Delta} c_{aa'}\mathbf a',
$$ 
where 
$$
c_{aa'}=\sum_{\gamma\in\Gamma(\mathbf a,\mathbf a')} w_\gamma.
$$
In particular, it follows from these definitions that for any $\mathbf a$, $\phi(\Delta(\mathbf a))=0$.

 The map $\phi$ is in some sense defined to be the flow along a vector field. Clearly $\partial^{\Delta}=\phi^\Delta\partial$.

\begin{thm}\label{thm:morse-complex}
The Morse complex is a chain complex: $(\partial^{\Delta})^2=0$, and $\phi^\Delta$ is a chain map which induces an isomorphism $H_*(C_\bullet)\overset{\cong}{\to} H_*(C_\bullet^{\Delta})$. 
\end{thm}
\begin{proof}
We prove all statements simultaneously by induction on the number of vectors in the vector field. The base case of our analysis will be one vector, say $\Delta(\mathbf a)=\mathbf b$, and $\partial \mathbf b=r\mathbf a+\cdots$. Consider the acyclic subcomplex $A=(0\to R(\mathbf b)\to R(\partial \mathbf b)\to 0)$. We get a quasi-isomorphism $C_\bullet\to C_\bullet/A$. Now as a free $R$-module, we have an isomorphism $C_\bullet/A\overset{\cong}{\to}C_\bullet^{\Delta}$ under the map defined on generators $[\mathbf b]\mapsto 0, [\mathbf a]\mapsto \mathbf a-\frac{1}{r}\partial \mathbf b$, and $[\mathbf c]\mapsto \mathbf c$ otherwise. The gradient paths in $C_\bullet$ are all of the form $\mathbf a,\mathbf b,\mathbf c$ where $\mathbf c$ is a nonzero term in $\partial \mathbf b$ not equal to $\mathbf a$. The weight of this path is the negative of the coefficient of $\mathbf c$ in $\partial \mathbf b$, divided by $r$. Hence $\phi(\mathbf a)=-\frac{1}{r}\partial \mathbf b+\mathbf a$. On the other hand $\phi(\mathbf b)=0$ and $\phi(\mathbf c)=\mathbf c$ for critical generators. So the isomorphism is given by $\phi$ as claimed.

The boundary $\partial$ on $C_\bullet/A$ induces the boundary operator $\phi\partial\phi^{-1}$ on $C_\bullet^{\Delta}$, and we must now determine its form.  Consider a critical generator $\mathbf c$. $\phi^{-1}(\mathbf c)=[\mathbf c]$. So $\phi\partial\phi^{-1}(\mathbf c)=\phi\partial \mathbf c=\partial^{\Delta} \mathbf c$ as desired.

Now suppose the theorem is true for gradient vector fields with $k$ vectors, and assume we now have one with $k+1$ vectors. Choose a vector, $\mathbf b=\Delta \mathbf a$.

Let $\Delta_k$ be the vector field consisting of the $k$ vectors aside from $(\mathbf a,\mathbf b)$. Inductively we have a quasi-isomorphism
$$
\phi^k:=\phi^{\Delta_k}\colon (C_\bullet,\partial)\to (C_\bullet^{\Delta_k},\partial^{\Delta_k})
$$
Now, we claim the pair $(\mathbf a,\mathbf b)$ still represents a vector in this Morse complex. We need only verify that $\partial^{\Delta_k}\mathbf b=r\mathbf a+\cdots.$ This follows from the nonexistence of closed gradient paths in the original vector field,
since  a $\partial$-gradient path from $\mathbf b$ to $\mathbf a$ aside from the path $(\mathbf b,\mathbf a)$ would combine with the vector $(\mathbf a,\mathbf b)$ to form a closed gradient path. Let $\Delta_1$ be the vector field with the single vector $(\mathbf a,\mathbf b)$ on the Morse complex $C_\bullet^{\Delta_k}$. Then $(C_\bullet^{\Delta_k})^{\Delta_1}=C_\bullet^\Delta$. 
 So we also have a quasi-isomorphism $$\phi^1=\phi^{\Delta_1}\colon(C_\bullet^{\Delta_k},\partial^{\Delta_k})\to(C_\bullet^{\Delta},\partial^{\Delta_1})$$
where we emphasize that $\partial^{\Delta_1}$ is defined to be the weighted sum of $\partial$-gradient paths alternating between $\partial^{\Delta_k}$ and $\Delta_1$. We now need to check that $\phi^\Delta=\phi^{1}\phi^{k}$, and that $\partial^{\Delta_1}=\partial^{\Delta}$. In the following calculations let $\mathfrak X^i=\mathfrak X^{\Delta_i}$, and let $\mathfrak X=\mathfrak X^\Delta$.
\begin{align*}
\phi^{1}\phi^{k}(\mathbf u)&=\phi_{1}\sum_{\mathbf v\in\mathfrak X^{k}}c^k_{uv}\mathbf v \\
&=\sum_{\mathbf v\in\mathfrak X^{k}}\sum_{\mathbf w\in\mathfrak X}c^k_{uv}c^1_{vw}\mathbf w 
\end{align*}
where $c^i_{uv}$ measures gradient paths with respect to $\Delta_i$.
So we need to check that  
$$
\forall \mathbf w\in\mathfrak X,\quad\quad\sum_{\mathbf v\in\mathfrak X^{k}}c^k_{uv}c^1_{vw}=c_{uw}.
$$
Note that $\mathfrak X=\mathfrak X^k\setminus\{\mathbf a,\mathbf b\}$, so we can write
\begin{align*}
\sum_{\mathbf v\in\mathfrak X^k}c^k_{uv}c^1_{vw}&=
\left(\sum_{\mathbf v\in\mathfrak X}c^k_{uv}c^1_{vw} \right) +
c^k_{ub}c^1_{bw}+c^k_{ua}c^1_{aw}\\
&=c^k_{uw}+c^k_{ua}c^1_{aw}
\end{align*}
The first term simplifies as indicated, because when $\mathbf{v},\mathbf{w}\in\mathfrak X$, $c^1_{vw}$ is only nonzero if $\mathbf v=\mathbf w$. The term $c^1_{bw}$ is zero because $\mathbf b$ being $\Delta_1(\mathbf a)$ does not begin any gradient paths.
The term $c^k_{uw}$ measures gradient paths from $\mathbf u$ to $\mathbf w$ which do not involve the vector $(\mathbf a,\mathbf b)$.
$c^k_{ua}$ measures all gradient paths that end in $\mathbf a$, and $c^1_{aw}$ represents all gradient paths (alternating between $\Delta_1$ and $\partial^{\Delta_k}$) from $\mathbf a$ to $\mathbf w$. There is only one vector in 
$\Delta_1$, and $\partial^{\Delta_k}$ measures alternating paths between $\partial$ and $\Delta_k$, so $c^1_{aw}$ represents all gradient paths with respect to $\Delta$ that start with the basis element $\mathbf a$ and end at $\mathbf w.$ So the product $c^k_{ua}c^1_{aw}$ measures all gradient paths that pass through the basis element $\mathbf a$.
Thus the sum measures all gradient paths from $\mathbf u$ to $\mathbf w$: $c^k_{uw}+c^k_{ua}c^1_{aw}=c_{uw}$.

Note that by induction $\partial^{\Delta_k}=\phi^k\partial$ and $\partial^{\Delta_1}=\phi^1 \partial^{\Delta_k}$. Hence $\partial^{\Delta_1}=\phi^1\phi^k \partial=\phi\partial=\partial^{\Delta}$.
\end{proof}

\subsection{Non-free chain complexes}
We will need to adapt the above construction to chain complexes which are not free. Indeed our tree complexes all have $2$-torsion, so we adapt the notion of a gradient vector field to the case where the chain groups consist of both $\Z$- and $\z$-summands. We replace the notion of ``basis of a free $R$-module" with the notion of ``minimal generating set," where some generators span copies of $\Z$ and some span copies of $\z$.
 We define a vector field as above except that a vector $(\mathbf a,\mathbf b)$ cannot
mix a $\z$-generator and a $\Z$-generator. 
  Then $C^\Delta_\bullet$ is defined to be the subgroup of the chain group $C_\bullet$ spanned by critical generators. (In particular, $\z$-generators remain $2$-torsion in the Morse complex.)
 Gradient paths are defined as above, and come in two types. A gradient path ending in $\mathbf{a}$ is called a $\z$-path if $\mathbf a$ is a $\z$-generator and is called a $\Z$-path if $\mathbf a$ is a $\Z$-generator. Notice that if a gradient path involves a $\z$ generator at some stage, then every subsequent generator in the path will be a $\z$-generator. In particular, a $\Z$-path will only consist of $\Z$-generators.
   The weight of a gradient path is defined as follows. If  is defined by the same formula as in the free case for $\Z$-paths: $w_\gamma= (-1)^{m-1} \frac{c_1\cdots c_{m-1}}{r_1\cdots r_{m-1}}\in \Z$. For $\z$-paths, $w_\gamma= c_1\cdots c_{m-1}\in \z$, where some of these coefficients may be in $\Z$, but are interpreted mod $2$. The $r_i$'s are omitted as they are all $\pm 1$, so are trivial mod $2$. The weights of $\partial$-gradient paths are defined similarly.
   
The flow $\phi^\Delta$ is defined as a sum of weights of gradient paths, as in the free case. With these definitions, the proof of Theorem~\ref{thm:morse-complex} then goes through with little modification in this particular non-free setting:

\begin{thm}\label{thm:morse-nonfree}
Let $C_\bullet$ be a chain complex where each $C_k$ is a finitely generated abelian group whose only torsion is $2$-torsion, and suppose that $\Delta$ is a gradient vector field in the above sense on $C_\bullet$. Then the Morse complex is a chain complex: $(\partial^{\Delta})^2=0$, and $\phi^\Delta$ is a chain map which induces an isomorphism $H_*(C_\bullet)\overset{\cong}{\to} H_*(C_\bullet^{\Delta})$.
\end{thm}


\section{Proof of Theorem~\ref{thm:vanish}}\label{sec:vanish-thm-proof}
Recall the statement of Theorem~\ref{thm:vanish}: $H_1(\mathbb L_{\bullet,n};\mathbb Z)=0$.
The proof is done in two stages by Propositions~\ref{prop:oddprime} and \ref{prop:2tor}, whose proofs occupy most of this section. 

Let $\mathbb L^{(2)}_{\bullet,n}\subset \mathbb L_{\bullet,n}$ be the subcomplex spanned by $2 \mathbb L_{\bullet,n}$ and symmetric degree $2$ trees as on the right side of Figure~\ref{hall2} which have no orientation-reversing symmetry. (The obvious symmetry turns out to be orientation preserving, because the trees being swapped have an even number of edges. See the discussion in the proof of Proposition~\ref{prop:oddprime}.) This is evidently a subcomplex since the boundary of such a tree is a multiple of $2$.

\begin{prop}\label{prop:oddprime}
For all $n$, $H_1(\mathbb L_{\bullet,n}; \zh)=0,$ and $H_1(\mathbb L^{(2)}_{\bullet,n};\Z)=0$.
\end{prop}

\begin{prop}\label{prop:2tor}
For all $n$, $H_1(\mathbb L_{\bullet,n};\mathbb Z_2)=0$.
\end{prop}

Before proving these propositions, we check that they imply Theorem~\ref{thm:vanish}.
We would like to use the universal coefficient theorem to conclude that $H_1(\mathbb L_{\bullet,n};\Z)=0$, but our chain groups are not free modules, so we pause to state a lemma that holds in this context. 
 For any chain complex $C_\bullet$ define 
$$H_k^{(2)}(C_\bullet)= \frac{Z_k\cap 2C_k}{2Z_k\cup (B_k\cap 2C_k)},$$
where $Z_k$ and $B_k$ are the submodules of cycles and boundaries, respectively.

\begin{lem}\label{lem:uct}
For all $k$, there is an exact sequence 
$$
0\to H_k^{(2)}(C_\bullet)\to \z\otimes H_k(C_\bullet)\to H_k(C_\bullet;\z).
$$ 
\end{lem}
\begin{proof}
The right-hand map is defined by $1\otimes[z]\mapsto [1\otimes z]$ for cycles $z$. Clearly the $[1\otimes z]$ is still a cycle. This is well-defined because boundaries $1\otimes \partial w$ map to boundaries $\partial(1\otimes w)$.  We claim that the left-hand map is an injection. Suppose a cycle $2u$ maps to $0$ in $\z\otimes H_k(C_\bullet)$. So $1\otimes 2u=\partial (1\otimes w)$, implying $1\otimes 2u=1\otimes \partial w$ in $\z\otimes C_\bullet$. Therefore, $\partial w = 2u+2z$, where $z$ is a cycle. Therefore $2u\in 2Z_k\cup (B_k\cap 2C_k)$, and hence equals $0$ in the domain. Finally, to see exactness at the middle, suppose that $1\otimes[z]\mapsto 0\in H_k(C_\bullet;\z)$. Then $1\otimes z=\partial (1\otimes w)$. Therefore $\partial w=z+2x$. Hence $z-\partial w\in 2C_\bullet$, and the homology class $[z]=[z-\partial w]$ is represented by a cycle in $2C_\bullet$.
\end{proof}

\begin{proof}[Proof of Theorem~\ref{thm:vanish}]
We apply Lemma~\ref{lem:uct}, using propositions~\ref{prop:oddprime} and \ref{prop:2tor}: Notice first that $H_1(\mathbb L^{(2)}_{\bullet,n};\Z)$ surjects onto $H_1^{(2)}(\mathbb L_{\bullet,n};\Z),$ because $\mathbb L_{1,n}^{(2)}=2\mathbb L_{1,n}$. Thus $\z\otimes H_1(\mathbb L_{\bullet,n};\Z)$ is trapped between two zero groups and is therefore zero. And tensoring with $\zh$ things are even easier, since $\zh\otimes H_1(\mathbb L_{\bullet,n})$ injects into $H_1(\mathbb L_{\bullet,n};\zh)$.\end{proof}

\subsection{Plan of the proofs of Propositions~\ref{prop:oddprime} and ~\ref{prop:2tor}}
Because the proofs of these two propositions are somewhat technical, we would like to give an overview of the structure of these proofs. Fix $n$, and to simplify notation let $\mathbb L_{\bullet}=\mathbb L_{\bullet,n}$.
The idea will be, with two different sets of coefficients, to construct a vector field $\Delta=\Delta_0\cup \Delta_1$, where $\Delta_i\colon\mathbb L_i\to\mathbb L_{i+1}$, which has no critical vectors in degree $1$ and then appeal to Theorem~\ref{thm:morse-complex} to conclude the degree $1$ homology vanishes.

First consider $\zh$ coefficients. Since $H_0(\mathbb L_{\bullet};\zh)\cong \zh\otimes{\sf L}_{n}$ has a well-known basis, called the Hall basis, our strategy in the $\zh$ case will be to define $\Delta_0(J)$ to be some nontrivial contraction of $J$ for every non-Hall tree $J$. In particular, we will define a ``Hall problem'' of a 
tree $J$ to be a place in $J$ where two Hall trees meet at a vertex, but their bracket is not itself Hall. 
$\Delta_0(J)$ then contracts an edge at the base of one of the two Hall trees. The resulting vector field has no closed gradient paths, in some sense because the Hall basis algorithm ``works.'' The way we rigorously prove it is to show that a natural ``Hall order" defined on trees always increases as one moves along a gradient path.
This is the most one can hope to do for $\Delta_0$, because Hall trees need to survive as a basis for $H_0(\mathbb L_{\bullet};\zh)$. 

All the trees in the image of $\Delta_0$ are not critical, so it suffices now to define $\Delta_1$ to be nonzero on all the other degree $1$ trees. In analogy with degree $0$, we say a tree is ${\rm Hall}_1$ if it is in the image of $\Delta_0$, because these are the trees on which $\Delta_1$ needs to vanish. We combinatorially characterize what it means to be ${\rm Hall}_1$, and define a ``${\rm Hall}_1$ problem'' to be one of an exhaustive list of ways that a tree can fail to be ${\rm Hall}_1$. Finally we define $\Delta_1$ for each different type of ${\rm Hall}_1$ problem as a certain contraction of an edge within the ${\rm Hall}_1$ problem. The resulting vector field is again shown to be gradient by arguing that the Hall order increases as one moves along gradient paths. This then proves the $\zh$ case, and in fact proves Proposition~\ref{prop:oddprime} since the chain complexes
$\zh\otimes\mathbb L_{\bullet}$ and $\Z\otimes \mathbb L^{(2)}_{\bullet}$ are isomorphic.

For $\z$ coefficients the argument is similar, except now we use the fact proven by Levine \cite{L3}, that $H_0(\mathbb L_{\bullet};\z)\cong\z\otimes{\sf L}'_n$ has a basis given by Hall trees plus trees $(H,H)$ where $H$ is Hall. We call the trees in this basis ${\rm Hall}'$ trees, and proceeding as before, we define $\Delta'_0$ by contracting ${\rm Hall}'$ problems. Defining a ${\rm Hall}'_1$ tree to be a tree in the image of $\Delta'_0$, a ${\rm Hall}'_1$ problem is one of an exhaustive list of ways that a tree can fail to be ${\rm Hall}'_1$. Then $\Delta'_1$ is defined for these different ${\rm Hall}'_1$ problems by contracting certain edges within the problems. We argue that the Hall order increases along gradient paths, except in one case, and using special arguments to take care of this case, this shows that there are no closed gradient paths.

\subsection{Proof of Proposition~\ref{prop:oddprime}}
We prove both cases simultaneously. Recall that the chain group $\mathbb L_{\bullet}$ is defined as the quotient of a free $\Z$ module of oriented trees by relations $(J,-\mathrm{or})=-(J,\mathrm{or})$. In particular, a tree $J$ will either generate a $\Z$- or a $\z$-summand depending on whether it has an orientation reversing automorphism, so that $\mathbb L_{\bullet}$ is a direct sum of copies of $\Z$ and $\z$. We claim that 
the nonzero trees in $\zh\otimes\mathbb L_{\bullet}$ and $\mathbb L^{(2)}_{\bullet}$ are the same.  Multiplying by $2$ and adjoining $1/2$ both have the effect of killing the $2$-torsion, and the additional symmetric trees in $\mathbb L^{(2)}_{2}$ are not $2$-torsion, so that this is indeed true.
 
To define a gradient vector field, first we need to specify a basis.
Choose orientations for each tree. Define the basis of the free $\zh$-module 
$\zh\otimes\mathbb L_{\bullet}$ to be the trees with specified orientations, except for nonzero symmetric degree $2$ trees as on the right of figure~\ref{hall2}. In that case the basis element is defined to be $1/2$ the given oriented tree. For the free $\Z$-module $\mathbb L^{(2)}_{\bullet}$, define the basis to be exactly twice the oriented trees just mentioned.  In fact $\zh\otimes\mathbb L_{\bullet}\cong \mathbb L^{(2)}_{\bullet}$ as chain complexes with these specified generators.
Thus it makes sense to construct a gradient vector field on both $\zh\otimes\mathbb L_{\bullet}$ and $\mathbb L^{(2)}_{\bullet}$ simultaneously. We will construct such a vector field with
no critical generators in degree $1$.  We will work with $\zh\otimes\mathbb L_{\bullet}$ but since the correspondence of bases respects the boundary operator, this will simultaneously prove the $\mathbb L^{(2)}_{\bullet}$ case.

The vector field will consist of two pieces $\Delta_{0}\colon \zh\otimes\mathbb L_{0}\to\zh\otimes\mathbb L_{1}$, and $\Delta_1\colon \zh\otimes\mathbb L_{1}\to\zh\otimes\mathbb L_{2}$. 
$\Delta_{0}$ is
constructed via the Hall Basis algorithm for the free Lie algebra. So we need to set up some machinery to explain this. Our presentation follows and expands upon \cite{Reu}.

Recall that for rooted trees $J_1$ and $J_2$, the rooted tree $(J_1,J_2)$ is defined by identifying the roots together to a single vertex, and attaching a new rooted edge to this vertex. Extend this notation to $(J_1,\ldots,J_k)$, which is defined by identifying the roots of all the trees $J_i$ to a single vertex, and attaching a new rooted edge to this vertex.
Thus the two trees in Figure~\ref{symfig} are notated $((1,2),(1,2))$ and $((1,2,3),(1,2,3))$, respectively.

If $J_1,\ldots, J_k$ are oriented, the tree $(J_1,\ldots,J_k)$ will be oriented by numbering the root edge first, then the edges of $J_1$, followed by the edges of $J_2$, etc.  

Given a rooted tree, suppose one deletes an internal vertex. Then there are multiple connected components, including the one containing the root. Any of the connected components that does not contain the root can itself be regarded as a rooted tree, by filling in the deleted vertex as a root. Such a tree will be called a \emph{full subtree.}

With this notation and terminology in hand, we can characterize those rooted trees which have orientation-reversing automorphisms, and are therefore zero when tensoring with $\zh$.
Indeed, any rooted tree is zero which contains a full subtree of the form $(J_1,\ldots, J_m)$ where $J_i=J_k$ for some $i\neq k$ with $J_i=J_k$ having an odd number of edges. For example, this will be true when $J_i$ is unitrivalent. Hence, over $\zh$ we may assume that for trees
in $\mathbb L_{0}$ and in $\mathbb L_{1}$ the emanating subtrees at every vertex are \emph{distinct}.
Exemplar trees with orientation preserving and reversing automorphisms are pictured in Figure~\ref{symfig}.
\begin{figure}
\begin{center}
$
\begin{minipage}{.6in}
\includegraphics[width=.6in]{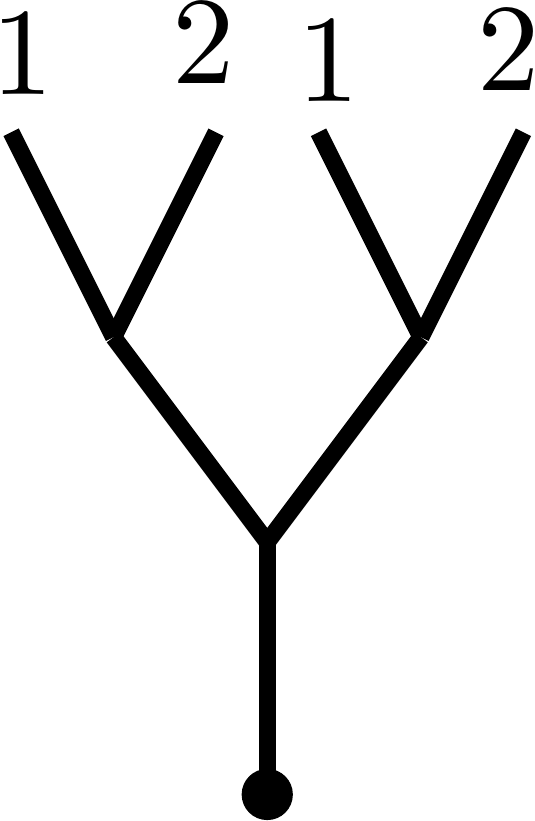}
\end{minipage}
=0\in\zh\otimes \mathbb L_{0,4}
\hspace{.5in}
\begin{minipage}{.6in}
\includegraphics[width=.6in]{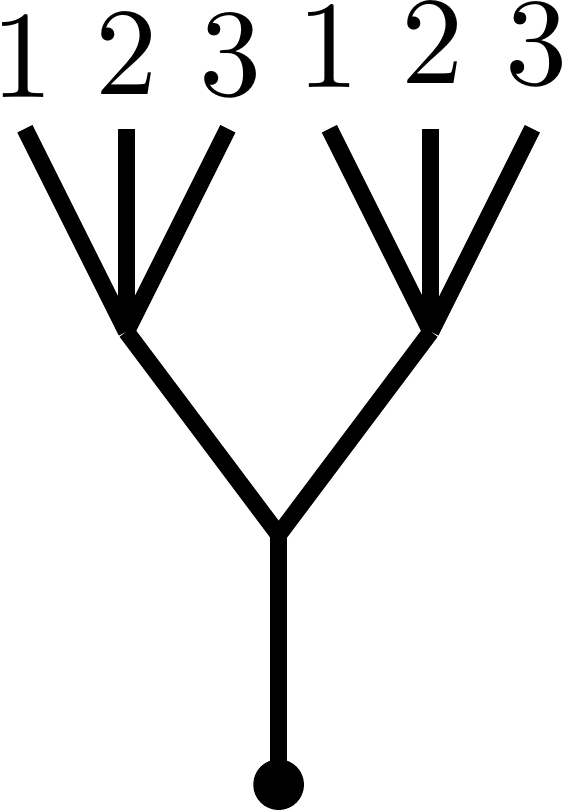}
\end{minipage}
\neq 0\in\zh\otimes \mathbb L_{2,6}
$
\end{center}
\caption{The tree $((1,2),(1,2))$ on the left has an orientation reversing symmetry, whereas the tree $((1,2,3),(1,2,3))$ on the right has an orientation preserving symmetry.}\label{symfig}
\end{figure}

\begin{defn}\label{def:order}
Define the \emph{weight} of a tree, denoted by $|J|$, to be the number of leaves (the number of univalent vertices not counting the root).
We recursively define an order relation, called the \emph{Hall order}, on (unoriented) labeled trees in the following way. 
\begin{enumerate}
\item Trees of weight $1$ are ordered by an ordering on the index set. 
\item If $|J|<|K|$, then $J\succ K$.
\item If $|J|=|K|$ and $J=(J_1,\ldots, J_\ell)$, $K=(K_1,\ldots,K_m)$ with $\ell<m$ then $J\succ K$.
\item If $|J|=|K|$ and $J=(J_1,\ldots,J_m)$, $K=(K_1,\ldots,K_m)$, assume that $J_1\preceq J_2\preceq\cdots \preceq J_m$ and $K_1\preceq K_2\preceq\cdots\preceq K_m$. Then compare the two trees lexicographically:
$J\prec K$ if an only if there is some index $\ell$, $1\leq \ell <m$, such that $J_i=K_i$ for all $i\leq \ell$ and $J_{\ell+1}\prec K_{\ell+1}$.
\end{enumerate}
\end{defn}

\begin{lem}\label{subtree}
Suppose one replaces a full subtree of a tree by a tree with increased Hall order. Then the new tree has larger Hall order than the original.
\end{lem}
\begin{proof}
If the new subtree tree has lower weight, then this is clear.  
If the new tree has the same weight, then recursively it will suffice to consider the replacement $(J_1,\ldots,J_\ell,\ldots, J_m)\mapsto (J_1,\ldots,J_\ell',\ldots,J_m)$ where the full subtree $J_\ell$ has been replaced by $J_\ell'$, with $J_\ell\prec J_\ell'$. By the lexicographic definition (case (iv) of Definition~\ref{def:order}), the whole tree is also larger with respect to $\prec$. 
\end{proof}

\begin{defn}\label{def:Hall-Hallproblem-contraction}
\begin{enumerate}
\item The set of \emph{Hall trees} is the subset of (unoriented) labeled rooted unitrivalent trees, defined recursively as follows. 
All weight $1$ trees are Hall. The bracket of two Hall trees $(H_1,H_2)$ with $H_1\prec H_2$ is Hall if and only if either $H_1$ is of weight $1$, or $H_1=(H',H'')$ with $H'\prec H''$ and $H''\succeq H_2$.
\item Given a tree $J$, a \emph{Hall problem} is a full subtree of $J$ of the form $(H_1,H_2)$ where $H_1\prec H_2$ are both Hall, but $H_1=(H',H'')$ with $H'\prec H''\prec H_2$.
\item The \emph{contraction of a Hall problem} is the tree obtained by replacing the full subtree $(H_1,H_2)$ by the full subtree $(H',H'',H_2)$.
\end{enumerate}
\end{defn}

The above definitions are made for unoriented trees, but we will often not distinguish between oriented and unoriented trees. For example, when we say that an oriented tree is Hall, we mean that the underlying unoriented tree is Hall. 

We now define a vector field $\Delta_0\colon \zh\otimes\mathbb L_{0}\to\zh\otimes\mathbb L_{1}$ as follows. 
If $H$ is a Hall tree, then $\Delta_0(H)=0$. Otherwise, suppose $J$ is a tree where there is at least one Hall problem. Define 
$$
\Delta_0(J)=\max\{J^c\,|\, J^c\text{ is the contraction of a Hall problem in }J\}.
$$ 

\begin{lem}~\label{lem:delta0gradient}
$\Delta_{0}$ is a gradient vector field.
\end{lem}
\begin{proof}
We need to check the following three conditions.
\begin{enumerate}
\item  $\Delta_0(J_1)\neq \Delta_0(J_2)$ for distinct $J_1,J_2$ which are non-Hall.
\item $\partial\Delta_0(J)=\pm J+\text{ other trees, }$ for $J$ non-Hall. 
\item All gradient paths terminate.
\end{enumerate}
All three of these will follow from the fact that the two expansions of $\Delta_0(J)$, other than $J$ itself, have larger Hall order, which we now check. Indeed, by Lemma~\ref{subtree}, we just need to show
$$
((H',H''),H_2)\prec ((H',H_2),H'') \text{ and } ((H',H''),H_2)\prec ((H'',H_2),H')
$$
Note we are assuming that $H'\prec H''\prec H_2$, corresponding to the Hall problem $((H',H''),H_2)$ in $J$. Then $(H',H_2)\prec H'\prec H''$, so the term $((H',H_2),H'')$ is in non-decreasing order. So applying the lexicographic definition, it suffices to show that $(H',H'')\prec (H',H_2)$ which follows from Lemma~\ref{subtree}. Now for the second inequality above, the two component trees of $((H'',H_2),H')$ may occur in either order, or even be equal. If they are equal, the tree is zero, and we don't need to consider it, but because we want to re-use this analysis in the case of $\z$ coefficients, we prefer not to use this fact. In any event, no matter what the order, it suffices to observe that $(H',H'')\prec (H'',H_2)$ and $(H',H'')\prec H'$. Thus the two expansions of $\Delta_0(J)$, other than $J$ itself, have larger Hall order, as claimed.

Condition (i) follows because $\Delta_0(J_1)=\Delta_0(J_2)$ implies that $J_2$ is another expansion of $\Delta_0(J_1)$ so that $J_1\prec J_2$. Symmetrically $J_2\prec J_1$, which is a contradiction. 
 Condition (ii) follows since $\partial\Delta_0(J)=J+J'+J''$ where $J'$ and $J''$ have increased Hall order.
Condition (iii) now follows since the Hall order increases as one flows along a gradient path, implying that all gradient paths must terminate.
\end{proof}

So far we have adapted the well-known Hall algorithm to the context of homological vector fields. Since the critical generators in degree zero are Hall trees, we have reproduced the standard fact that the free Lie algebra $\zh\otimes{\sf L}_n\cong H_0(\mathbb L_{\bullet};\zh)$ is generated by Hall trees. Since these trees actually form a basis, to make further progress in killing degree $1$ generators we will need to construct a vector field $\Delta_{1}\colon\zh\otimes\mathbb L_{1}\to\zh\otimes \mathbb L_{2}$, extending $\Delta_0$.

\begin{defn}
A nonzero tree, $H\in\zh\otimes \mathbb L_{1}$, is said to be ${\emph Hall_1}$ if $H=\Delta_{0}(J)$ for some tree $J\in\zh\otimes\mathbb L_{0}$.
\end{defn}

We now characterize ${\rm Hall}_1$ trees. 

\begin{lem}\label{lem:Hall1}
A tree $H\in\zh\otimes \mathbb L_{1}$ is ${\rm Hall}_1$ if and only if both of the following two conditions hold.
\begin{enumerate}
\item It contains a full subtree $(A,B,C)$, where $A\prec B\prec C$ are all Hall, such that either $|A|=1$, or $A=(A',A'')$ with $A'\prec A''$ and $A''\succeq B$.
\item The tree obtained by replacing $(A,B,C)$ with $((A,B),C)$ is nonzero, and $H$ is the largest among all contractions of Hall problems of this expanded tree.
\end{enumerate}
\end{lem}
\begin{proof}
The first condition says exactly that $(A,B,C)$ is the contraction of a Hall problem $((A,B),C)$, using the fact that $(A,B)$ is Hall. Moreover the other two expansions $((A,C),B)$ and $((B,C),A)$ are not Hall problems, so to check whether $H$ is in the image of $\Delta_0$ we need only check that it is the maximal contraction of this expanded tree, as per the second statement. 
\end{proof}

A ${Hall}_1$ \emph{problem} of a nonzero tree $J$, will be defined by negating the characterization of the previous lemma. Cases (i) and (ii) of the following Definition~\ref{def:hall1problem} are the two ways that condition (i) of the previous Lemma can fail, and cases (iii) and (iv) are the two ways that condition (ii) of the previous Lemma can fail.
\begin{defn}\label{def:hall1problem}  A \emph{Hall}$_1$ \emph{problem} in a nonzero tree $J\in\zh\otimes\mathbb L_{1}$ is defined to be any one of the following situations:
\begin{enumerate}
\item $J$ contains the full subtree $(A,B,C)$ where at least one of $A,B,C$ fails to be Hall. So there is a Hall problem in one of these trees.
\item  $J$ contains the full subtree $(A,B,C)$ where $A\prec B\prec C$ are all Hall, but $A=(A',A'')$ with $A'\prec A''\prec B$. 
\item $J$ contains a full subtree $(A,B,C)$, where $A\prec B\prec C$ are all Hall, such that either $|A|=1$, or $A=(A',A'')$ with $A'\prec A''$ and $A''\succeq B$; and the expanded tree containing $((A,B),C)$ is nonzero, but $J$ is not the largest of all contractions of the  expanded tree.
\item $J$ contains a full subtree $(A,B,C)$, where $A\prec B\prec C$ are all Hall, such that either $|A|=1$, or $A=(A',A'')$ with $A'\prec A''$ and $A''\succeq B$; but the expanded tree containing $((A,B),C)$ is zero. This means that $((A,B),C)$ occurs in the tree $J$, and there is a symmetry in the expanded tree exchanging this for the $((A,B),C)$ expansion of $(A,B,C)$. Such a $J$ is pictured in the middle of Figure~\ref{hall2}, where the expanded tree on the left is $0$.
\end{enumerate}
\end{defn}
\begin{figure}[ht!]
$$
\xymatrix{
0=\begin{minipage}{1in}\includegraphics[width=1in]{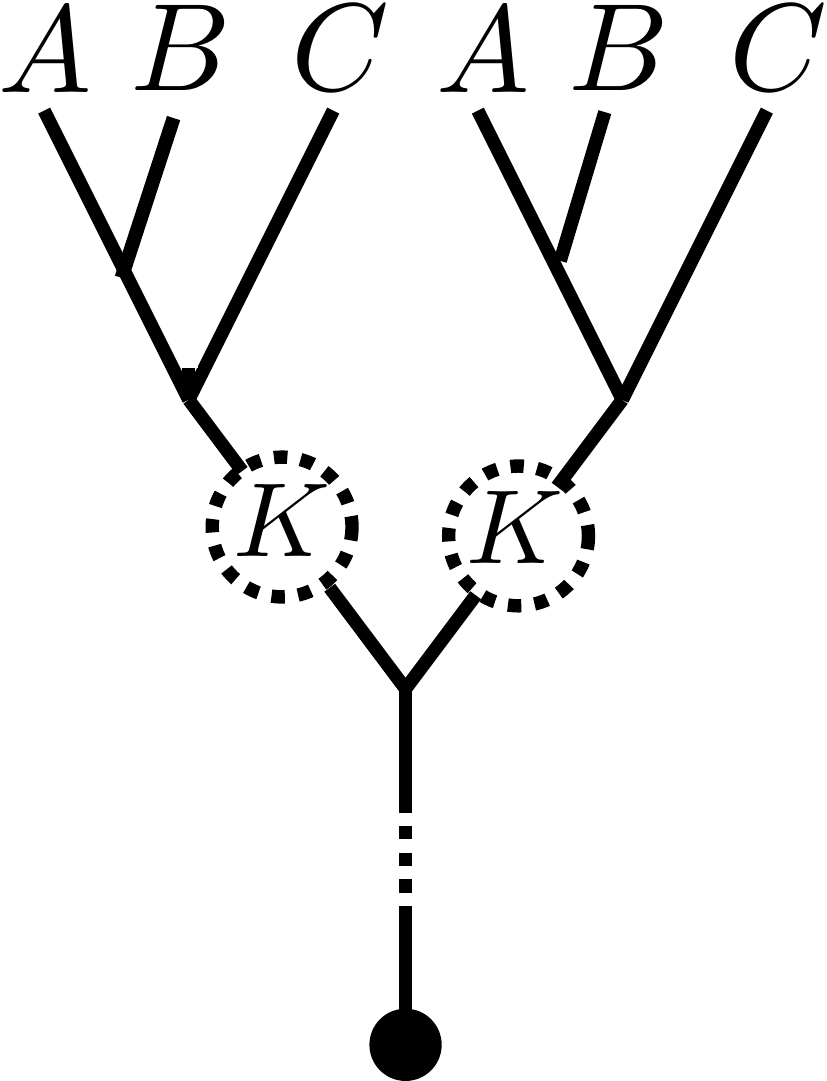}\end{minipage}\ar@{|->}[r]|-{??}&
\begin{minipage}{1in}\includegraphics[width=1in]{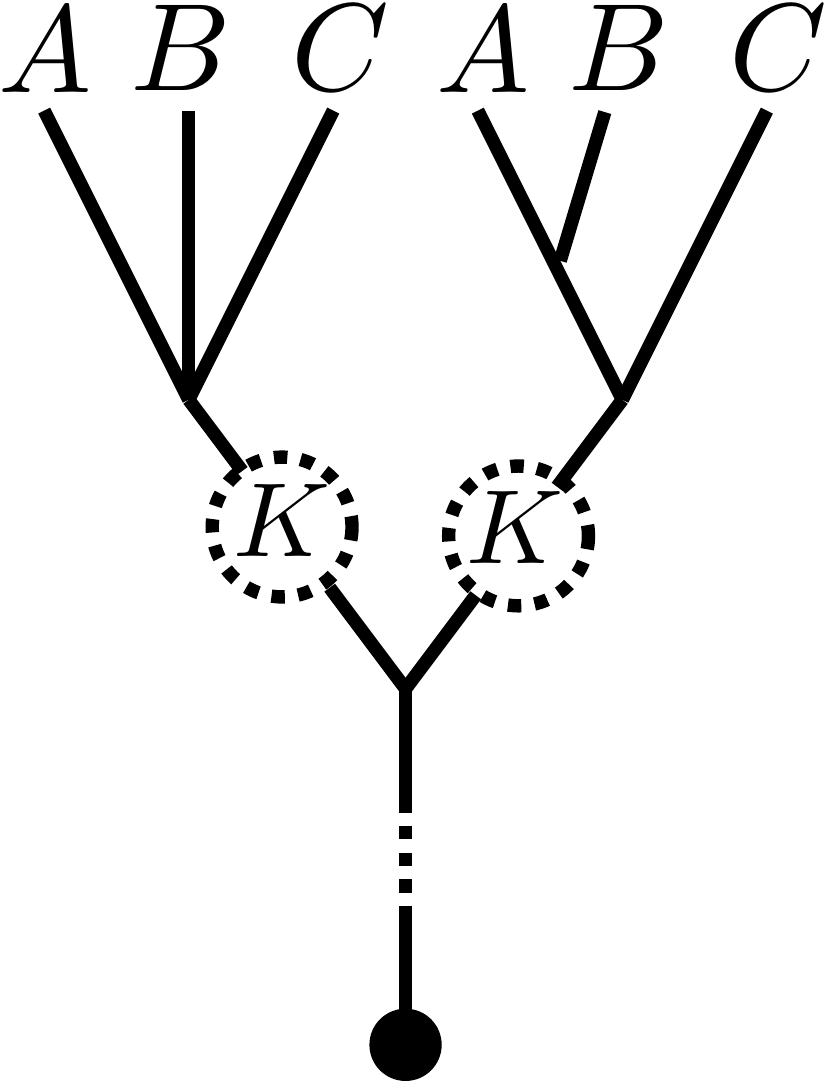}\end{minipage}
\ar@{|->}[r]^{\Delta_1}&\frac{1}{2}\begin{minipage}{1in}\includegraphics[width=1in]{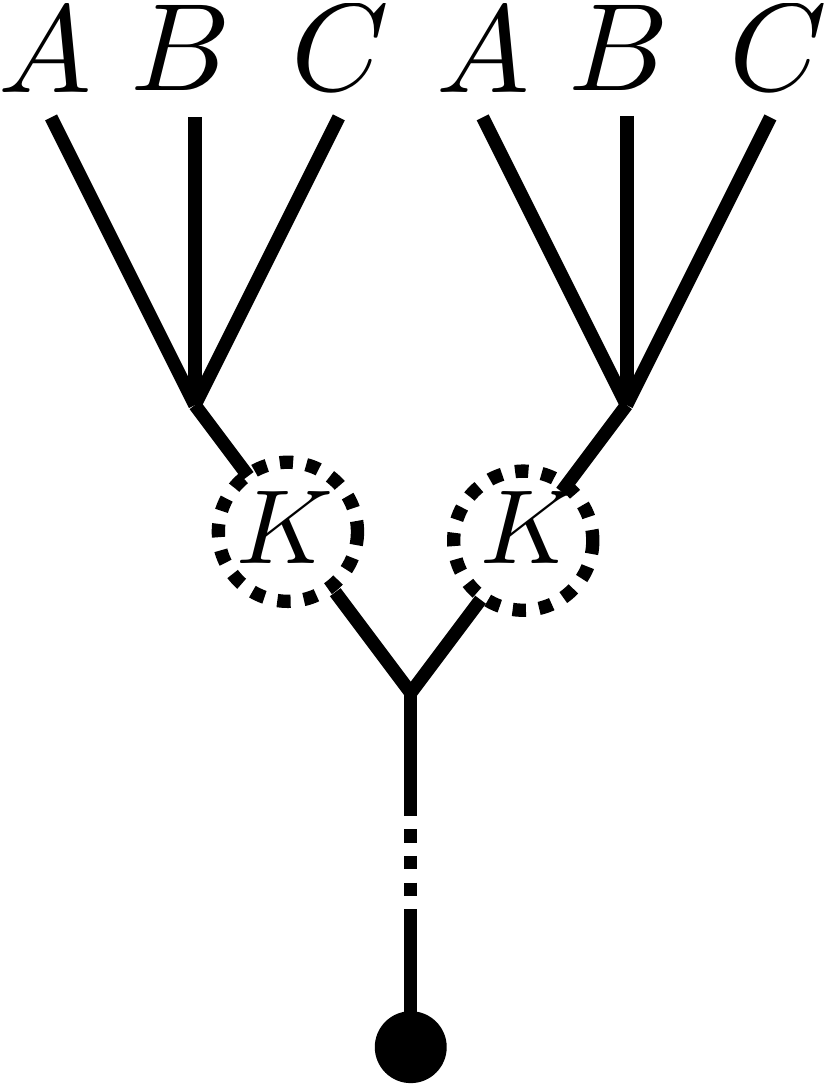}\end{minipage}
}
$$
\caption{The tree, $J$, in the middle has a type (iv) ${\rm Hall}_1$ problem, since the expanded tree on the left is equal to zero in 
$\zh\otimes\mathbb L_{0}$.  The image of $J$ under $\Delta_1$ is pictured on the right. (The sub-trees indicated by the dotted arcs are all trivalent.)}\label{hall2}
\end{figure}

Now we define the vector field $\Delta_1$. 
\begin{defn}\label{def:Delta1}
If a tree is  ${\rm Hall}_1$ we define $\Delta_1$ to be zero. (As we need to, in order to make $\Delta_1$ disjoint from $\Delta_0$.) Otherwise, referring to Definition~\ref{def:hall1problem}, we consider the four types of ${\rm Hall}_1$ problems: 
\begin{enumerate}
\item If $J$ has a type (i) ${\rm Hall}_1$ problem, then define $\Delta_1(J)$ to contract some Hall problem in the trees $\{A,B,C\}$. 

\item If $J$ has a type (ii) ${\rm Hall}_1$ problem, then define $\Delta_1(J)$ to contract the base edge of $A$. 

\item If $J$ has a type (iii) ${\rm Hall}_1$ problem, to define $\Delta_1(J)$, consider the expanded tree $J^e$ containing $((A,B),C)$. Then $\Delta_0(J^e)$ contracts a different Hall problem in $J^e$ (which is maximal among all contractions) than the one that contracts to $J$. Define $\Delta_1(J)$ by contracting the image in $J$ of this other Hall problem in $J^e$. 

\item If $J$ has a type (iv) ${\rm Hall}_1$ problem, define $\Delta_1(J)$ to be the contraction of the other copy of $((A,B),C)$ with a coefficient of $1/2$ (Figure~\ref{hall2}). This contracted tree has a symmetry, but it is orientation-preserving because it exchanges two trees which have a single $4$-valent vertex and are otherwise unitrivalent. As mentioned in the introductory remarks to the proof, $1/2$ the contracted tree is indeed a basis element.
\end{enumerate}
\end{defn}

\begin{lem}
$\Delta_1$ is a gradient vector field.
\end{lem}
\begin{proof}
As we did for $\Delta_{0}$, we will show that if $J$ is not ${\rm Hall}_1$, the expansions of $\Delta_1(J)$ distinct from $J$ have increased Hall order. So assume $J$ is not ${\rm Hall}_1$. That means there is a ${\rm Hall}_1$ problem in $J$.

We analyze the cases separately. 

Case (i). We have a tree $J$ with a full subtree $(A,B,C)$ which has a Hall problem in one of the $A,B,C$ trees. Contracting this Hall problem and then applying $\partial$ without backtracking will either expand the Hall problem into one of the two other trees besides $J$, which we already showed increases the Hall order, or will expand the $(A,B,C)$ vertex, which increases the Hall order because $(A,B,C)\prec (U,V)$  if $|(A,B,C)|=|(U,V)|$.

Case (ii). This case has the most numerous collection of terms. The full subtree $((A',A''),B,C)$ with $A'\prec A''\prec B\prec C$ gets contracted to $(A',A'',B,C)$ and we must show that all terms of $\partial(A',A'',B,C)$ except $((A',A''),B,C)$ increase in Hall order. If the $4$-valent vertex gets pushed up the tree, like in $((A',A'',B),C)$, then the Hall order goes up, so we need only consider the $5$ possibilities where this doesn't happen:
$$((A',B),A'',C),((A',C),A'',B),((A'',B),A',C),$$
$$((A'',C),A',B),((B,C),A',A'')$$
The components of these trees are not necessarily in order, although they definitely are in the first two trees, and in fact the smallest tree in each of the first two triples has larger Hall order than $(A',A'')$. For the third tree, it could be that either of $(A'',B)$ or $A'$ is larger, and they could even be equal, however it suffices to observe that both are greater than $(A',A'')$. Similar remarks hold for the fourth tree. For the fifth tree, we know $A'\prec A''$, but we have no information on $(B,C)$. So it will suffice to observe that $(A',A'')\prec (B,C)$ and $(A',A'')\prec A'$.

Case (iii). Let $J$ be a nonzero unitrivalent tree with at least two Hall problems, and suppose that $J^{1}$ is the largest contraction of a Hall problem and that $J^2$ is the contraction of another one, so that $J^1\succeq J^2$, and by definition $J^1=\Delta_0(J)$. Case (iii) is the case of the tree $J^2$, assuming it is distinct from $J^1$.
 Now by definition, $\Delta_1(J^2)=J^{12}$, the tree with both Hall problems collapsed. Now we wish to analyze the trees appearing in $\partial(J^{12})$ distinct from $J^2$. Expanding the first vertex yields the tree $J^2$ and two other trees with larger Hall order. Expanding the second vertex yields $J^1$ and two other trees larger than $J^1$. Since $J_1\succ J_2$, this shows the Hall order has increased for all of these trees.
  
Case (iv). Here when we expand the tree with two copies of $(A,B,C)$, we will get twice the original tree $J$, which is why we needed to divide by $2$. The other trees are all increased with respect to $\prec$.
\end{proof}

Thus we have constructed a gradient vector field $\Delta_{0}\cup\Delta_1$ on $\zh\otimes\mathbb L_{0}$ such that every basis element in degree $1$ is either in the range of $\Delta_{0}$ or the support of $\Delta_1$. So in the Morse complex, there are no nonzero degree $1$ chains, implying the first homology is zero. Incidentally, this also reproduces the classical result that the Hall trees (in degree $0$) are independent and form a basis for $\zh\otimes{\sf L}_n\cong H_0(\mathbb L_{\bullet};\zh)$.

This concludes the proof of Proposition~\ref{prop:oddprime}.

\subsection{Proof of Proposition~\ref{prop:2tor}}
The proof is very similar to the proof of Proposition~\ref{prop:oddprime}, but
now trees with orientation-reversing automorphisms are no longer equal to zero in $\z\otimes\mathbb L_{\bullet}$. We therefore adapt the vector field to take account of these, in some sense generalizing Levine's quasi-Hall basis algorithm \cite{L3}. Define the specified basis for $\z \otimes \mathbb L_{\bullet}$ to consist of unoriented trees. This makes sense because tensoring with $\z$ erases the orientation data.

\begin{defn}
\begin{enumerate}
\item A tree in $\z \otimes \mathbb L_{\bullet}$ is a \emph{Hall}$'$ tree if it is either a ${\rm Hall}$ tree as defined previously in Definition~\ref{def:Hall-Hallproblem-contraction} for the $\zh$-coefficients case, or a tree of the form $(H,H)$ where $H$ is a ${\rm Hall}$ tree.
\item A \emph{Hall problem} in a tree $J$ in $\z \otimes \mathbb L_{\bullet}$ is either a ${\rm Hall}$ problem as defined in Definition~\ref{def:Hall-Hallproblem-contraction}, or it is a full subtree of $J$ of the form $(H,H)$, for a Hall tree $H$. Call this latter type of problem a \emph{symmetric Hall problem}.
\item A \emph{Hall}$'$ \emph{problem} in a tree $J$ is  a ${\rm Hall}$ problem which is not of the form $(H,H)$ where 
$J=(H,H)$ and $H$ is Hall.
\item The \emph{contraction} of a ${\rm Hall}'$ problem is defined as in Definition~\ref{def:Hall-Hallproblem-contraction}, with the contraction of  a symmetric Hall problem defined by contracting the root edge of the $(H,H)$ subtree.
\end{enumerate}
\end{defn}

These definitions were set up so that
$$ J\text{ is not Hall } \Leftrightarrow J\text{ has a ${\rm Hall}$ problem.}$$
$$ J\text{ is not Hall}' \Leftrightarrow J\text{ has a ${\rm Hall}'$ problem.}$$

Now, similarly to the $\zh$-coefficients case, $\Delta'_0$ is defined to vanish on ${\rm Hall}'$ trees; and for $J$ not ${\rm Hall}'$, 
$\Delta'_0(J)$ is defined to be the maximal contraction of a ${\rm Hall}'$ problem in $J$. 

\begin{lem}
$\Delta'_0\colon \z\otimes\mathbb L_{0}\to \z\otimes\mathbb L_{1}$ is a gradient vector field.
\end{lem}
\begin{proof}
As in the proof of Proposition~\ref{prop:oddprime}, we argue that terms in $\partial\Delta_0'(J)$, other than $J$ itself, have increased Hall order.
For ${\rm Hall}'$ problems which are not symmetric Hall problems, we have already argued this in the proof of Lemma~\ref{lem:delta0gradient}. For a symmetric Hall problem we contract $((H,H),B)$ to $(H,H,B)$. Notice that $\partial(H,H,B)=2((H,B),H)+((H,H),B)=(H,H,B)$. 
Thus the only nonzero term in $\partial J$, where $J$ is the tree containing $(H,H,B)$, is the original tree. Thus there are no gradient paths involving the vector $((H,H),B)\mapsto (H,H,B)$.
\end{proof}

\begin{defn}
A \emph{Hall}$'_1$ tree is defined to be a non-zero tree in the image of $\Delta'_0$.
\end{defn}

\begin{lem}\label{lem:char}
A tree $H$ is $\text{\rm Hall}^\prime_1$ if and only if both of the following two conditions hold:
\begin{enumerate}
\item $H$ either
\begin{enumerate} 
\item contains a full subtree $(A,B,C)$, where $A\prec B\prec C$ are all Hall, such that either $|A|=1$ or $A=(A',A'')$ with $A'\prec A''$ and $A''\succeq B$;
\item or it contains a full subtree $(A,A,B)$ where $A$ is Hall.
\end{enumerate}
\item Let $H^e$ be the tree where the full subtree $(A,B,C)$ is expanded to $((A,B),C)$ in case (a) above, or where the full subtree $(A,A,B)$ is expanded to $((A,A),B)$ in case (b). Then $H$ is the largest tree among all contractions of ${\rm Hall}'$ problems in $H^e$.
\end{enumerate}
\end{lem}
\begin{proof}
The first condition characterizes being the contraction of a ${\rm Hall}'$ problem and the second makes sure that it is a maximal contraction. Because all trees are nonzero in the complex $\z\otimes\mathbb L_{\bullet}$, there is no longer the subtlety that the expanded tree could be zero (as was the case in Lemma~\ref{lem:Hall1}).
\end{proof}
\begin{defn}\label{def:char}
Negating the  characterization from Lemma~\ref{lem:char}, we define a \emph{Hall}$'_1$ \emph{problem} in a tree $J\in\z\otimes\mathbb L_{1}$  to be any one of the following situations. Assume $J$ contains the full subtree $(A,B,C)$ with $A\preceq B\preceq C$.
\begin{enumerate}
\item One of the trees $A,B,C$ has a ${\rm Hall}$ problem:
\begin{enumerate}
\item $A=B$ and $A$ is not Hall.
\item $A\prec B=C$ and $B$ is not Hall.
\item $A\prec B\prec C$ and one of $A,B,C$ is not Hall.
\end{enumerate}
 \item $A,B,C$ are Hall with $A\prec B\prec C$ but $A=(A',A'')$ with $A'\prec A''\prec B$.
 \item $A=B$ are Hall or $B=C$ are Hall or $A\prec B\prec C$ with $|A|=1$ or $A''\succeq B$, but this is not the maximal contraction of a ${\rm Hall}'$ problem in the expanded tree.
\end{enumerate}
\end{defn}

\begin{defn}\label{def:Delta1prime}
Define $\Delta'_1\colon\z\otimes\mathbb L_{1}\to\z\otimes\mathbb L_{2}$ to vanish on $\textrm{Hall}'_1$ trees. Otherwise, referring to  Definition~\ref{def:char},
$\Delta'_1$ is defined for each case as follows: 
\begin{enumerate}
\item \label{def-item:exceptional}
\begin{enumerate}
\item Define $\Delta'_1$ to contract some ${\rm Hall}$ problem in $A$.
\item  Define $\Delta'_1$ to contract some ${\rm Hall}$ problem in $B$.
\item Define $\Delta'_1$ to contract some ${\rm Hall}$ problem in $A$, $B$, or $C$.
\end{enumerate}
Each of these three subcases includes the possibility that the contraction of the ${\rm Hall}$ problem contracts the root edge of $A$, $B$ or $C$. For example
$A$ might be a tree of the form $(H,H)$, and $\Delta'_1$ is then defined by contracting the root edge of $A$ to give a full subtree of the form $(H,H,B,C)$.
\item As in Definition~\ref{def:Delta1}, define $\Delta'_1$ to contract the root edge of $A$. 
\item As in Definition~\ref{def:Delta1}, define $\Delta'_1$ to contract the other, maximal, ${\rm Hall}'$ problem.
\end{enumerate}
\end{defn}

\begin{lem}\label{lem:deg1grad}
$\Delta'_1\colon\z\otimes\mathbb L_{1}\to\z\otimes\mathbb L_{2}$ is a gradient vector field.
\end{lem}
This will complete the proof of Proposition~\ref{prop:2tor}, as we have constructed a gradient vector field $\Delta'_0\cup\Delta'_1$ on $\z\otimes\mathbb L_{\bullet}$ with  no critical basis elements in degree $1$. 

\subsection{Proof of Lemma~\ref{lem:deg1grad}}
We need to verify three things as in the proof of Lemma~\ref{lem:delta0gradient}: 
\begin{enumerate}
\item  $\Delta'_1(J_1)\neq \Delta'_1(J_2)$ for distinct $J_1,J_2$ that are not ${\rm Hall}_1'$.
\item $\partial\Delta'_1(J)=J+\text{ other trees}$, for $J$ not ${\rm Hall}_1'$.
\item All gradient paths terminate.
\end{enumerate}

We proceed similarly to that proof by arguing that in all but one case, the trees in the sum $\partial\Delta'(J)-J$ all have larger Hall order than $J$. This exceptional case is dealt with by supplemental arguments that verify the three conditions. 

For type (i) ${\rm Hall}'_1$ problems, we are contracting a ${\rm Hall}'$ problem ``above" the $4$-valent vertex. If this problem is a ${\rm Hall}$ problem, we have already done the required analysis in the $\zh$-coefficients proof to show that the other expansions have increased Hall order.  If it is a contraction of an $(H,H)$ the same analysis applies with the exception of the case singled out in item~(\ref{def-item:exceptional}) of the definition of $\Delta'_1$, namely that the full subtree  $((H,H),B,C)$  contracts to $(H,H,B,C)$. For type (ii) and (iii) ${\rm Hall}'_1$ problems, the $\zh$-coefficient analysis remains valid, and the other expansions have increased Hall order.

Now that we know that the Hall order increases along gradient flows in all but the case where the full subtree  
$((H,H),B,C)$ contracts to $(H,H,B,C)$, we proceed to verify conditions (i), (ii) and (iii) for this exceptional case. First we check that condition (i) holds. By the proof of Lemma~\ref{lem:delta0gradient}, the only potential difficulty is when at least one of $J_1$ and $J_2$  is in the exceptional case. Say $J_1$ contains the full subtree $((H,H),B,C)$, which contracts to $(H,H,B,C)$. There are  four nonzero terms in $\partial\Delta'_1(J_1)$: $J_1$, and three trees containing the full subtrees $(H,H,(B,C))$, $((H,H,B),C)$ and $((H,H,C),B)$. With the exception of $J_1$, none of these trees contains a ${\rm Hall}'_1$ problem that contracts to $(H,H,B,C)$. Thus $(H,H,B,C)$ is the image of a unique tree under $\Delta'_1$. 

Condition (ii) also needs only to be checked in the exceptional case, as it automatically follows in the other cases from the fact that the other terms of $\partial\Delta'_1(J)$ have increased Hall order. So we must
verify that the nonzero terms of $\partial(H,H,B,C)$ contain only one copy of $((H,H),B,C)$. The only other term that could be equal to $((H,H),B,C)$  is $(H,H,(B,C))$, which would imply that $B=C=H$. In this case $((H,H),H,H)$ is the contraction of a ${\rm Hall}'$ problem in $((H,H),(H,H))$, and $\Delta'_1$ is either $0$, if the tree containing $((H,H),(H,H))$ is in the image of $\Delta'_0$, or $\Delta'_1$ contracts some other Hall problem elsewhere in the tree. So the hypothesis that
$\Delta'_1$ contracts $((H,H),B,C)$ to $(H,H,B,C)$ is not satisfied.

Finally, we tackle (iii). We will show that any gradient path that starts in the exceptional case can never return to its starting point. This implies that a closed gradient path will not contain these exceptional cases, and then the fact that the Hall order increases along gradient paths away from the exceptional cases implies there are no closed paths. 

Define a \emph{descendant} of a tree $T$, to be a tree $S$ such that there is a gradient path starting at $T$ and ending at $S$.
 
\begin{lem}\label{lem:merit}
Suppose $\Delta'_1$ contracts the full subtree $((H,H),B,C)$ to $(H,H,B,C)$. Every descendant of the tree containing $((H,H),B,C)$ is either in the image of $\Delta'_0$ or contains a full subtree of the form $(H,H,A)$ for some tree $A$. 
\end{lem}
\begin{cor}\label{cor:gradient}
Any gradient path starting with the tree containing $((H,H),B,C)$ never returns to $((H,H),B,C)$.
\end{cor}
\begin{proof}[Proof of Corollary~\ref{cor:gradient}]
This obviously follows from Lemma~\ref{lem:merit} unless $B=C=H$. However we argued just above in the proof of Lemma~\ref{lem:deg1grad} that this case violates the hypothesis that $\Delta'_1$ contracts $((H,H),B,C)$ to $(H,H,B,C)$.
\end{proof}
\begin{proof}[Proof of Lemma~\ref{lem:merit}]
First, consider descendants which are connected by a length $2$ gradient path, i.e., descendants which are terms of $\partial\Delta-\id$ applied to the tree containing $((H,H),B,C)$. There are three such descendants, and they contain trees of the following forms:
$$(H,H,(B,C)),((H,H,B),C),((H,H,C),B)$$ 
All three contain a subtree of the appropriate type. So now assume inductively that $J$ is a $k$th descendant containing a full subtree of the form $(H,H,A)$ and consider what trees are connected to $J$ by a length $2$ gradient path starting at $J$. Since $(H,H,A)$ is the contraction of a ${\rm Hall}'$ problem, $\Delta'_1(J)$ will either be $0$, if it was the maximal contraction, or it will contract a different ${\rm Hall}'$ problem which gives the maximal contraction, possibly in the tree $A$. 
So $\Delta'_1(J)$ contains one of the following subtrees, depending on the nature of the contracted ${\rm Hall}'$ problem:
\begin{enumerate} 
\item $(H,H,A)$ if the ${\rm Hall}'$ problem giving the maximal contraction is not in $A$.
\item $(H,H,A^c)$ if the ${\rm Hall}'$ problem giving the maximal contraction is in $A$ but this is not a contraction of the root. Here $A^c$ represents the contraction of this ${\rm Hall}'$ problem in $A$.
\item $(H,H,K,K)$ if $A=(K,K)$ and $\Delta'_1$ contracts the root of $A$.
\end{enumerate}
 Applying $\partial$ without backtracking will always yield a subtree $(H,H,B)$ except for the trees containing, respectively, $((H,H),A)$, $((H,H),A^c)$ and $((H,H),K,K)$. However these trees are all in the image of $\Delta'_0$ as they are each by definition the tree with the maximal contraction of a ${\rm Hall}'$ problem.
\end{proof}


\section{Proof of Theorem~\ref{lem:reducedT}}~\label{sec:proof}
Recall the statement of Theorem~\ref{lem:reducedT}: $\bbeta_\bullet$ induces an isomorphism
$\bar\eta\colon{\cT}_n\overset{\cong}{\longrightarrow} H_1(\overline{\mathbb L}_{\bullet,n+2}).$ Recalling that ${\cT}_n$ is isomorphic to $H_0(\mathbb T_{\bullet,n+2};\Z)$ (Proposition~\ref{prop:isomorphisms}),
we will construct a gradient vector field $\Delta$ on the chain complex $\overline{\mathbb L}_{\bullet,n+2}$,
and then show that the resulting chain map $\phi^\Delta\circ \bbeta$ induces an isomorphism 
$$\xymatrix{
H_0(\mathbb T_{\bullet,n+2};\Z)\ar@{->}[r]^{\bar\eta}\ar@/_2pc/[rr]^{\cong}&H_1(\overline{\mathbb L}_{\bullet,n+2};\Z)\ar[r]^{\phi^{\Delta}}_\cong&H_1(\overline{\mathbb L}^\Delta_{\bullet,n+2};\Z)
}$$
implying $\bar\eta$ is an isomorphism as desired. (Where $\phi^\Delta$ is an isomorphism by Theorem~\ref{thm:morse-nonfree}.)

To define $\Delta$ we begin by noting that via the operation of ``removing the rooted edge'' each tree 
$J\in\overline{\mathbb L}_{\bullet,n+2}$ defines a unique tree $t\in\mathbb T_{\bullet,n+2}$ which we call the \emph{underlying tree} of $J$. Specifically, if the root of $J$ is adjacent to an internal vertex of valence greater than $3$, then $t$ is gotten by deleting the root and (the interior of) its edge but leaving the internal vertex. If the root of $J$ is adjacent to a trivalent vertex, then $t$ is gotten by deleting the root and (the interior of) its edge but converting the resulting $2$-valent vertex into a non-vertex point of $t$. This operation in the case where the root is adjacent to a trivalent vertex is succinctly described using the notion of the \emph{inner product} 
$\langle J_1,J_2 \rangle$ of two rooted trees $J_1$ and $J_2$, which is the unrooted tree defined by identifying the roots of $J_1$ and $J_2$ to a single non-vertex point. (The orientation of $\langle J_1,J_2 \rangle$ is given by numbering the middle edge first, followed by the edges of $J_1$ and then the edges of $J_2$ in the orderings prescribed by their orientations.)  Then $\langle J_1,J_2 \rangle$ is the underlying tree of the bracket $(J_1,J_2)$.

Next choose a fixed basepoint for every isomorphism class of tree $t\in\mathbb T_{\bullet,n+2}$ in the following way. If an internal vertex is fixed by the symmetry group of $t$, choose the basepoint to be such a vertex.
Otherwise, by the following lemma, the tree is of the form $t=\langle T, T\rangle$ for some tree $T$. 
In this case choose the basepoint of $t$ to be the midpoint of the middle edge joining the two copies of $T$.
\begin{lem}
Every $t\in\mathbb T_{\bullet,n+2}$ is either of the form $\langle T, T\rangle$ or it has an internal vertex fixed by its symmetry group.
\end{lem}
\begin{proof}
The barycenter of a tree $t$ is defined as the midpoint of a maximal geodesic, and is uniquely defined. 
Note that the symmetry group of the tree fixes this barycenter. If the barycenter is in the middle of an edge, and one endpoint of that edge is not fixed by the symmetry group of the tree, then the tree is of the form  $t=\langle T, T\rangle$ for some tree $T$. 
\end{proof}

Now define $\Delta$ as follows: If the root of $J$ is adjacent to an internal vertex of valence greater than $3$, then $\Delta(J)=0$. If $J=(T,T)$ then we also define $\Delta(J)=0$. Otherwise, $\Delta(J)$ is defined by ``sliding'' the rooted edge of $J$ away from the 
basepoint of the underlying tree until it attaches to the next internal vertex:
$$
\xymatrix{
\begin{minipage}{1in}
\includegraphics[width=1in]{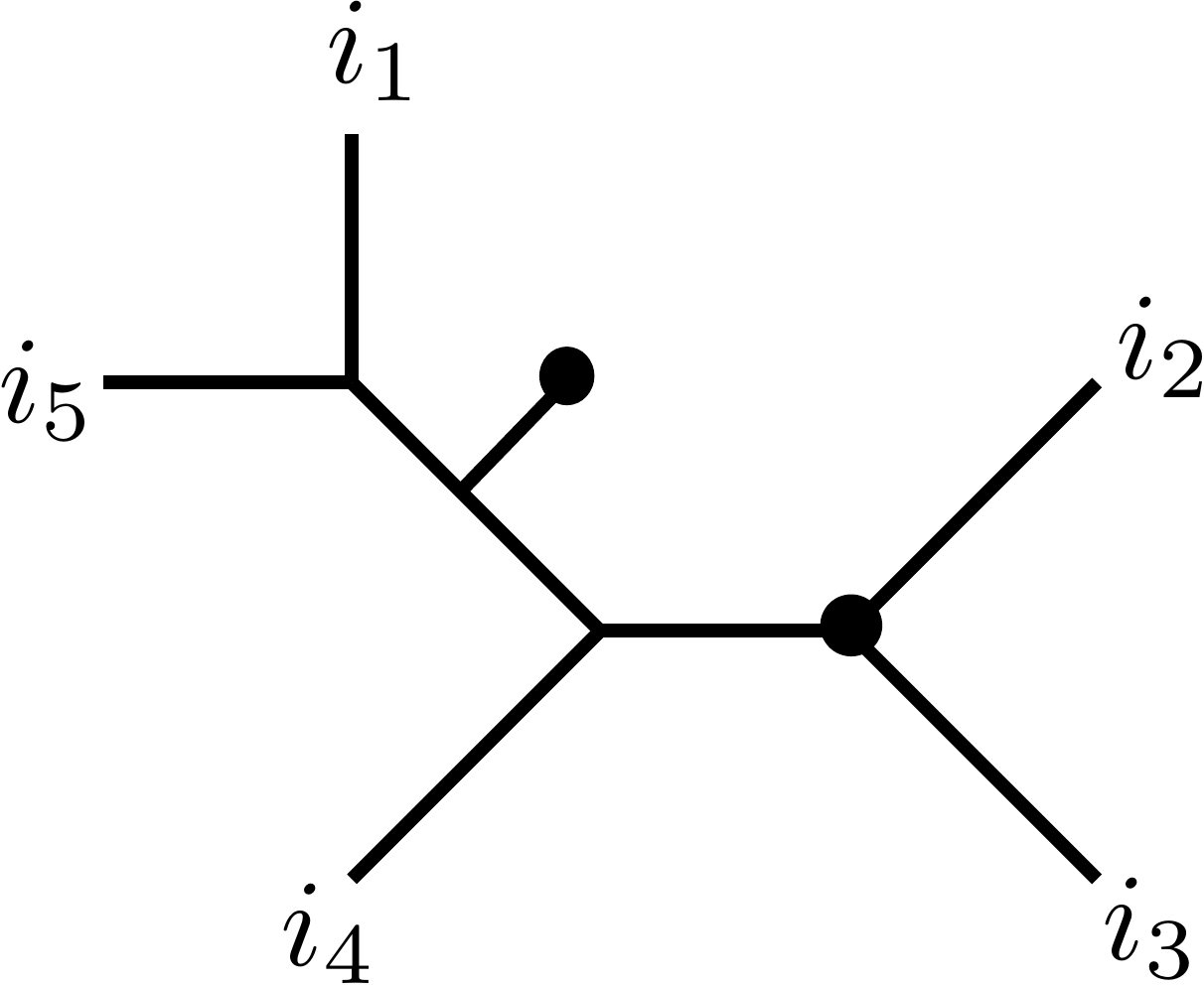}
\end{minipage}\ar@{|->}[r]^-{\Delta}&\begin{minipage}{1in}
\includegraphics[width=1in]{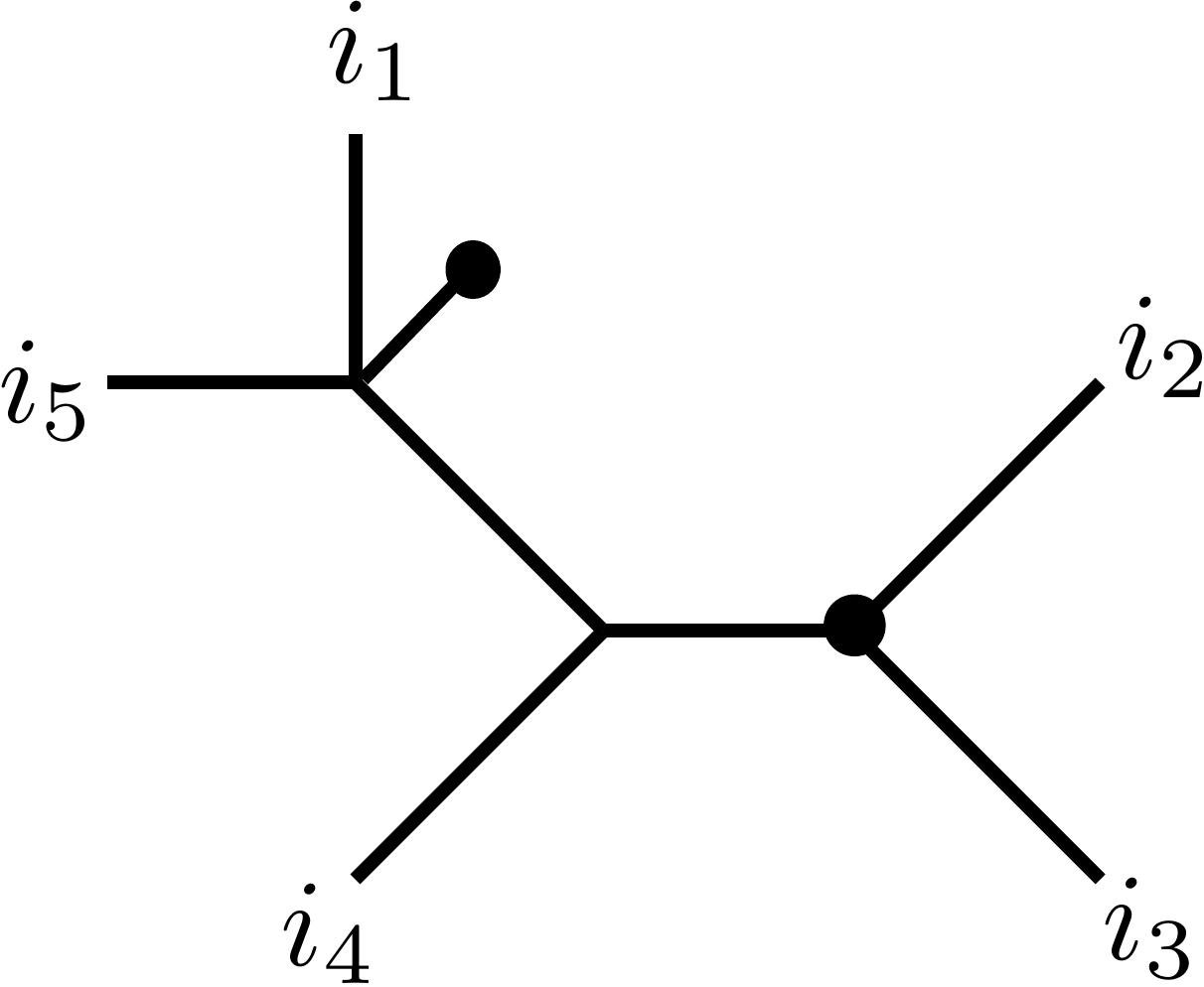}
\end{minipage}}
$$

 \begin{lem}
$\Delta\colon\overline{{\mathbb L}}_{\bullet,n+2}\to \overline{{\mathbb L}}_{\bullet+1,n+2}$ is a well-defined gradient vector field.
\end{lem}
\begin{proof}
Note that $\Delta$ is well-defined as a map: Since there are no trees of the form $(i,T)$ in $\overline{\mathbb L}_{\bullet,n+2}$, there will always be a neighboring internal vertex to slide to as required. Also $\partial(\Delta(J))=\pm J$ plus other terms not equal to $J$ since other expansions of $\Delta(J)$ where the root is adjacent to a trivalent vertex have the root further away from the basepoint.

We must check that $\Delta$ does not map any $\Z$-trees to $\z$-trees, and does not map any $\z$-trees to $\Z$-trees.
We only need to check the cases where $J=(J_1,J_2)$ has the root adjacent to a trivalent vertex. Note that every symmetry of $J$ obviously fixes the root.  If $J$ has a symmetry which flips the two outgoing trees $J_1$ and $J_2$, then $J_1=J_2$ and $\Delta(J)$ was defined to be $0$. If $J_1\neq J_2$, then any symmetry of $J$ restricts to the identity on the edge of the underlying tree to which the root edge attaches, so when we slide the root along this edge to get 
$\Delta(J)$, the symmetry is still there. If there are no non-trivial symmetries of $J$, then sliding the root edge to the neighboring internal vertex can not create a symmetry because such a symmetry would not fix the basepoint on the underlying tree, contradicting our choice of basepoints.

Let us analyze what a gradient path is in this vector field. A non-vanishing application of $\Delta$ pushes the root away from the basepoint. Now the only terms of $\partial(\Delta(J))-J$ on which $\Delta$ will evaluate nontrivially are those where the root has been pushed by 
$\partial$ onto another edge of the underlying tree that is further away from the basepoint. Repeated applications of $\Delta$ moves the rooted edges in such terms further and further away, until they eventually reach univalent edges, where the trees are zero in 
$\overline{\mathbb L}_{\bullet,n+2}$. Thus there are no closed gradient paths and $\Delta$ is a gradient field.
\end{proof}
Next we analyze the critical generators. Given a tree $t\in\mathbb T_{\bullet,n+2}$, let $t^b$ denote the tree with a rooted edge attached to the basepoint. Also let $\langle T,T\rangle^s$ denote the tree where a rooted edge is attached to one of the endpoints of the central edge.
\begin{lem}
The Morse complex is generated by the following two  types of trees.
\begin{enumerate}
\item Trees of the form $t^b$, for $t\in\mathbb T_{\bullet,n+2}$
\item Trees of the form $\langle T, T\rangle^s$.
\end{enumerate}
\end{lem}
\begin{proof}
If $J$ is a tree where the root edge is adjacent to a trivalent vertex, then $\Delta$ is nonzero, so $J$ is not critical. If the root attaches to a higher-valence vertex $v$ away from the basepoint of the underlying tree $t$, then consider the tree $J'$ where the root edge attaches to the middle of the edge of $t$ adjacent to $v$ that is closer to the basepoint. Then, in most cases, $\Delta(J')=J$. The one case that is ruled out is the case when $J'=(T,T)$, where $\Delta$ was defined as $0$. This explains why the tree $\langle T,T\rangle^s$ is critical. Finally, if the root edge attaches to the basepoint, $t$ is not in the image of $\Delta$, so it is critical.
\end{proof}

Recall now the chain map $\bbeta_\bullet\colon\mathbb T_\bullet\to\overline{\mathbb L}_{\bullet+1}$ from Lemma~\ref{lem:eta-chain-map} defined by summing over attaching a root to all internal vertices. We calculate $\phi^\Delta\circ\bbeta_\bullet$ as follows.
If $t$ is not of the form $\langle T,T\rangle$, then $\phi^\Delta(\bbeta(t))=t^b\in\overline{\mathbb L}^\Delta_{\bullet,n+2}$. This is because $\Delta$ vanishes on the tree summands of $\bbeta(t)$, so $\phi^\Delta(\bbeta(t))$ is the sum of critical generators in $\bbeta(t)$. On the other hand, suppose $t=\langle T,T\rangle$ where
$T$ is of even degree. Then $\phi^\Delta(\bbeta(\langle T,T\rangle))=2\langle T,T\rangle^s$. If $T$ has odd degree, then $\bbeta(\langle T,T\rangle)=0$, because there is an orientation reversing automorphism exchanging the two endpoints of the central edge.
 
Now consider the chain map $\phi^\Delta\circ\bbeta$ which will have both a kernel $\Ker_{\bullet}$ and a cokernel $\Cok_{\bullet}$:
$$0\to \Ker_{\bullet}\to\mathbb T_{\bullet}\to \overline{\mathbb L}^\Delta_{\bullet+1}\to \Cok_{\bullet}\to 0$$
Let us now analyze the kernel and cokernel. First we set up some convenient notation. A tree denoted by $A$ must have an orientation reversing automorphism. A tree denoted by $K$ must have no orientation reversing 
automorphism. A tree denoted by $J$ may or may not have one.

\begin{lem}
 The cokernel can be written as follows. \begin{align*}
&\Cok_{4i+2}=\Z\{\langle K,K\rangle^s \,|\, \deg(K)=2i+1\}\oplus\z\{\langle A,A\rangle^s \,|\, \deg(A)=2i+1\}\\
&\Cok_{4i+1}=\Z\{(K,K)\,|\, \deg(K)=2i+1\}\oplus 
\z\{(A,A)\,|\,\deg(A)=2i+1\}\\
&\Cok_{4i}=\z\{\langle J,J\rangle ^s\,|\,\deg(J)=2i\}\\
&\Cok_{4i-1}=  \z\{ (J,J)\,|\, \deg(J)=2i\}
\end{align*}
Moreover $\Cok_\bullet$ is an acyclic complex.
\end{lem}

\begin{proof}
The critical trees that are not hit by $\phi^\Delta\bbeta_\bullet$ are of two kinds: $(T,T)$ and $\langle T,T\rangle^s$. No multiple of $(T,T)$ is in the image,
whereas $2\langle T,T\rangle^s$ \emph{is} in the image if and only if $T$ has even degree. 

If $T$ is of odd degree and itself has an orientation-reversing automorphism, then both $(T,T)$ and $\langle T,T\rangle^s$ will be $2$-torsion, accounting for the $\z$-summands in degrees $4i+1$ and $4i+2$ above. If $T$ has no such automorphism, then neither does $\langle T,T\rangle^s$ nor $(T,T)$, accounting for the $\Z$-summands. Finally, when $T$ has even degree, $2\langle J,J\rangle^s$ is in the image of $\phi^\Delta\bbeta_\bullet$ and $(T,T)$ is $2$-torsion, accounting for the remaining terms above,

 Notice that $\partial^\Delta\langle J,J\rangle^s=(J,J)$, because one term of $\partial$ is critical and equal to 
$(J,J)$ and the other terms lead to gradient flows that push the root away from the basepoint and which eventually terminate in $0$. Thus $\Cok_{4i}\to \Cok_{4i-1}$ is a direct sum of acyclic complexes of the form $\z\to\z$. Similarly $\Cok_{4i+2}\to \Cok_{4i+1}$ is a direct sum of acyclic complexes either of the form $\z\to\z$ or $\Z\to\Z$. Therefore $\Cok_\bullet$ is acyclic.
\end{proof}

Now we turn to an analysis of $\Ker_\bullet$. Again a tree called $A$ must have an orientation reversing automorphism.
\begin{lem}
The kernel can be written as follows.
 \begin{align*}
&\Ker_{4i+2}=\z\{\langle T,T\rangle\,|\, \deg(T)=2i-1\}\\
&\Ker_{4i+1}=0\\
&\Ker_{4i}=\z\{\langle A,A\rangle \,|\,\deg(A)=2i\}\\
&\Ker_{4i-1}= 0
\end{align*}
\end{lem}
\begin{proof}
Clearly $\phi^\Delta\circ\bbeta_\bullet$ is injective away from symmetric trees $\langle T,T\rangle$. If the degree of $T$ is odd, these are all in the kernel. If the degree of $T$ is even, then $\langle T,T\rangle\mapsto 2\langle T,T\rangle^s$, and so is nonzero unless $\langle T,T\rangle$ is $2$-torsion, implying that $T$ has an orientation-reversing automorphism. 
\end{proof}
We are interested in establishing that $\bbeta_0$ induces an isomorphism on homology. This will follow because $\Ker_0$ is generated by trees which are zero in $\mathcal T_n$. (In degree $0$, $A$ will contain a subtree of the form $(J,J)$. Apply IHX to the base edge of this copy of $(J,J)$ to see that $\langle A,A\rangle$ is $0$ in $\mathcal T_n$.)

Formally, we argue by splitting the exact sequence into two short exact sequences:
$$
\xymatrix{
\Ker_i\ar@{>->}^\iota[r]&\mathbb T_{i}\ar[rr]^{\phi^\Delta\bbeta_i}\ar@{->>}[dr]&&\overline{\mathbb L}^\Delta_{i+1}\ar@{->>}[r]&\Cok_i\\
&&\mathbb T_{i}/\im \iota\ar@{>->}[ur]&&
}
$$
Then because the inclusion $\iota_0\colon H_0(\Ker_\bullet)\to H_0(\mathbb T_\bullet)$ induces the zero map, we have that $H_0(\mathbb T_\bullet)\cong H_0(\mathbb T_\bullet/\im \iota)$. On the other hand, the cokernel is acyclic, so that $H_i(\mathbb T_\bullet/\im \iota)\cong H_{i+1}(\overline{\mathbb L}^\Delta_\bullet)$ for all $i$.
Therefore 
$$\mathcal T_n\cong H_0(\mathbb T_\bullet)\cong H_1(\overline{\mathbb L}^\Delta_\bullet)\cong
H_1(\overline{\mathbb L}_\bullet)\cong{\sf D}'_n.$$
 This completes the proof of Theorem~\ref{lem:reducedT}.

{\bf Remark:} It follows from the above discussion that $\bbeta_\bullet$ does not induce an isomorphism  $H_i(\mathbb T_\bullet)\to H_{i+1}(\overline{\mathbb L}_\bullet)$ for arbitrary $i$. Indeed, at the next degree there is an exact sequence
$$0\to H_1(\mathbb T_\bullet)\to H_2(\overline{\mathbb L}_\bullet)\to\z\{\langle A,A\rangle\,|\,\deg(A)=0\}\to 0$$
 This demonstrates a failure of surjectivity of $\bbeta_1$. In general $\bbeta$ can also fail to be injective. For example, $\langle(1,1,2),(1,1,2)\rangle\in\mathbb T_2$ represents a nontrivial $2$-torsion homology class, which can be easily checked since there is only one degree $3$ tree with the same signature.
This tree is clearly in the kernel of $\bbeta_2$.  
  
\section{Comparing filtrations of the group of homology cylinders}\label{sec:filtrations}
Let $\Sigma_{g,1}$ denote the compact orientable surface of genus $g$ with one 
boundary component. A \emph{homology cylinder} over $\Sigma_{g,1}$ is a compact $3$ manifold $M$ which is homologically equivalent to the cylinder $\Sigma_{g,1}\times[0,1]$, equipped with standard parameterizations of the two copies of $\Sigma_{g,1}$ at each ``end." 
Two homology cylinders $M_0$ and $M_1$ are said to be \emph{homology cobordant} if there is a compact oriented $4$-manifold $W$ with $\partial W=M_0\cup_{\Sigma_{g,1}} (-M_1)$, such that the inclusions $M_i\hookrightarrow W$ are homology isomorphisms. This defines an equivalence relation on the set of homology cylinders.
Let $\mathcal H_g$ be the set of homology cylinders up to homology cobordism over $\Sigma_{g,1}$. $\mathcal H_g$ is a group via the ``stacking" operation. 

Adapting the usual string link definition, Garoufalidis and Levine \cite{GL} introduced an Artin-type representation
$\sigma_n\colon \mathcal{H}_g\to A_0(F/F_{n+2})$ where $F$ is the free group on $2g$ generators, and $A_0(F/F_{n+1})$ is the group of automorphisms $\phi$ of $F/F_{n+1}$ such that  $\phi$ fixes the product $[x_1,y_1]\cdots [x_g,y_g]$ modulo $F_{n+1}$. Here $\{x_i,y_i\}_{i=1}^g$ is a standard symplectic basis for $\Sigma_{g,1}$.  The Johnson (relative weight) filtration of $\mathcal H_g$ is defined by $\mathbb {J}_n=\Ker \sigma_n$.  Define the associated graded group ${\sf J}_n=\mathbb J_n/\mathbb J_{n+1}$. Levine shows that ${\sf J}_n\cong{\sf D_n}$.

On the other hand, there is a filtration related to Goussarov-Habiro's theory of finite type 3-manifold invariants. We define the relation of $A_n$-equivalence to be generated by the following move: $M\sim_n M'$ if $M'$ is 
diffeomorphic to $M_C$, for some connected clasper $C$ with $n$ nodes. Let $\mathbb Y_n$ be the subgroup of $\mathcal H_g$ of all homology cylinders $A_n$-equivalent to the trivial one, and let ${\sf Y}_n=\mathbb Y_n/\sim_{n+1}$. Rationally, Levine showed the associated graded groups for these two filtrations are the same, and are even classified by the tree group $\mathcal T_{n}$:

\begin{thm}[Levine] \label{thm:levine} There is a commutative diagram\\
\centerline{$\xymatrix{
\cT_{n}\ar@{->>}^{\theta_n}[r]\ar@/_1pc/[rrr]_{\eta_{n}}&{\sf Y}_n\ar[r]&\mathcal {\sf J}_n\ar[r]^\cong &{\sf D}_n
}.$} All of these maps are rational isomorphisms.
\end{thm}

The story is more subtle over the integers. Levine conjectured the statements in the following theorem, which are straightforward consequences of the fact that $\eta'$ is an isomorphism.
\begin{thm}\label{thm:filtration}
There are exact sequences:
$$0\to{\sf Y}_{2n}\to {\sf J}_{2n}\to \z\otimes{\sf L}_{n+1} \to 0\,\,\, n\geq 1$$
$$\z^m\otimes {\sf L}_{n} \to{\sf Y}_{2n-1}\to {\sf J}_{2n-1}\to 0\,\,\, n\geq 2$$
\end{thm}

Levine did not conjecture that the map $\z^m\otimes {\sf L}_{n} \to{\sf Y}_{2n-1}$ is injective, and in fact it is \emph{not} injective, basically because the framing relations discussed in \cite{CST1} are also present in this context. 
As we prove in \cite{CST5}, for odd numbers of the form $4n-1$ this allows us to get a sharp answer  to what the kernel of ${\sf Y}_{4n-1}\to {\sf J}_{4n-1}$ is, while for odd numbers of the form $4n+1$ we determine it up to $\z\otimes {\sf L}_{n+1}$:

\begin{thm}[\cite{CST5}] \label{thm:filt} There are exact sequences
\begin{enumerate}
\item $0\to {\sf Y}_{2n}\to {\sf J}_{2n}\to \mathbb Z_2\otimes{\sf L}_{n+1}\to 0$\hspace{1em}$n\geq 1$
\item  $0\to \mathbb Z_2\otimes{\sf L}_{2n+1}\to {\sf Y}_{4n-1}\to{\sf J}_{4n-1}\to 0$\hspace{1em}$n\geq 1$
\item $0\to{\sf K}^{\sf Y}_{4n+1}\to{\sf Y}_{4n+1}\to {\sf J}_{4n+1}\to 0$\hspace{1em}$n\geq 0$, where the kernel ${\sf K}^{\sf Y}_{4n+1}$ fits into the exact sequence $\mathbb Z_2\otimes {\sf L}_{n+1}\overset{a_{n+1}}\to{\sf K}^{\sf Y}_{4n+1}\to \mathbb Z_2\otimes {\sf L}_{2n+2}\to 0$.
\end{enumerate}
\end{thm}
The calculation of the kernel ${\sf K}^{\sf Y}_{4n+1}$ is thus reduced to the calculation of $\Ker(a_{n+1})$. This is the precise analog of the question in the Whitney tower world of whether $\alpha_{n+1}$ is injective and how nontrivial are the higher-order Arf invariants.

\begin{conj}
The homomorphisms $a_{n+1}$ are injective for all $n\geq 0$, implying that there is an exact sequence $0\to \mathbb Z_2\otimes {\sf L}'_{2n+2}\to{\sf Y}_{4n+1}\to {\sf J}_{4n+1}\to 0$.
\end{conj}

To prove Theorem~\ref{thm:filtration} we will use the following exact sequences
$$0\to \z^m\otimes {\sf L}_n\to{\sf D}'_{2n-1}\to {\sf D}_{2n-1}\to 0$$
$$0\to {\sf D}'_{2n}\to {\sf D}_{2n}\to\z\otimes{\sf L}_{n+1}\to 0$$
proven by Levine in \cite{L2,L3}, and the commutative diagram of Theorem~\ref{thm:levine}.
\begin{proof}[Proof of Theorem~\ref{thm:filtration}]

In the even case, we get the following diagram:
$$\xymatrix{
&&&{\sf D}'_{2n}\ar@{>->}[d]\\
\cT_{2n}\ar@{->>}[r]_{\theta_{2n}}\ar[rrru]_{\eta'}^{\cong}&{\sf Y}_{2n}\ar[r]&{\sf J}_{2n}\ar[r]^\cong &{\sf D}_{2n}\ar@{->>}[d]\\
&&&\z\otimes {\sf L}_{n+1}
}
$$
from which it follows that the map ${\sf Y}_{2n}\to\mathcal {\sf J}_{2n}$ is injective with cokernel $\z\otimes {\sf L}_{n+1}$. Indeed after making various identifications using the isomorphisms in the diagram, we get a commutative diagram
$$\xymatrix{
\cT_{2n}\ar@{>->}[r]\ar@{->>}[d]^{\theta_{2n}}&\mathcal {\sf J}_{2n}\ar@{->>}[r]&\z\otimes{\sf L}_{n+1}\\
{\sf Y}_{2n}\ar[ur]&&
}
$$
where the top row is exact. By commutativity of the triangle, $\theta_{2n}$ is an isomorphism.
  
Turning to the odd case, we have a commutative diagram:
  $$\xymatrix{
  &&&\z^m\otimes{\sf L}_n\ar@{>->}[d]\\
&&&{\sf D}'_{2n-1}\ar@{->>}[d]\\
\cT_{2n-1}\ar@{->>}[r]_{\theta_{2n}}\ar[rrru]_{\eta'}^{\cong}&{\sf Y}_{2n-1}\ar[r]&\mathcal {\sf J}_{2n-1}\ar[r]^\cong &{\sf D}_{2n-1}\\
}
$$
which collapses, after identifications, to the following diagram, where the top row is exact.
$$
\xymatrix{
\z^m\otimes{\sf L}_n\ar@{>->}[r]&\mathcal T_{2n-1}\ar@{->>}[r]\ar@{->>}[d]_{\theta_{2n-1}}&\mathcal {\sf J}_{2n-1}\\
&{\sf Y}_{2n-1}\ar[ur]
}
$$
Thus the map   ${\sf Y}_{2n-1}\to\mathcal {\sf J}_{2n-1}$ is surjective, and $\z^m\otimes{\sf L}_n$ maps onto the kernel. 
\end{proof}


\end{document}